\documentclass[a4paper,reqno]{amsart}
\pdfoutput=1
\usepackage{filecontents}
\usepackage{iftex}
\ifPDFTeX
   \usepackage[utf8]{inputenc}
   \usepackage[T1]{fontenc}
   \usepackage{lmodern}
   \usepackage{amsmath,amssymb,amsthm}
\else
   \ifXeTeX
     \usepackage{fontspec}
   \else 
     \usepackage{unicode-math}
     \usepackage{luatextra}
   \fi
   \defaultfontfeatures{Ligatures=TeX}
\fi

\usepackage{mathtools}
\usepackage{thmtools}
\usepackage{xparse}

\newcommand{\R}{\mathbb{R}}
\newcommand{\N}{\mathbb{N}}
\usepackage{bbm}

\renewcommand{\P}{\mathbb{P}}
\renewcommand{\d}{\,d}
\newcommand{\id}[1]{1_{#1}}
\newcommand{\restr}[2]{#1\!_{\restriction #2}}

\DeclarePairedDelimiter\abs{\lvert}{\rvert}%
\DeclarePairedDelimiter\norm{\lVert}{\rVert}%
\makeatletter
\let\oldabs\abs
\def\abs{\@ifstar{\oldabs}{\oldabs*}}
\let\oldnorm\norm
\def\norm{\@ifstar{\oldnorm}{\oldnorm*}}
\makeatother

\newcommand{\fullstop}{\text{ .}}

\NewDocumentCommand{\lcro}{d<>m}{\IfValueTF{#1}{#1}{\left}[ #2 \IfValueTF{#1}{#1}{\right})}

\newtheorem{theorem}{Theorem}
\numberwithin{theorem}{section}
\newtheorem{lemma}[theorem]{Lemma}

\theoremstyle{definition}
\newtheorem{definition}[theorem]{Definition}
\newtheorem{remark}[theorem]{Remark}

\usepackage{csquotes}

\usepackage[scr=pxtx]{mathalfa}
\usepackage{relsize}
\def\fcmp{\mathbin{\raise 0.6ex\hbox{\oalign{\hfil$\scriptscriptstyle \mathrm{o}$\hfil\cr\hfil$\scriptscriptstyle\mathrm{9}$\hfil}}}}
\let\mcmp\fcmp

\usepackage{tikz}
\usepackage{enumitem}
\usepackage[normalem]{ulem}
\setlist[enumerate]{label=(\arabic*)}

\usepackage[hidelinks]{hyperref}
\renewcommand{\d}{\,\mathrm d}

\newcommand{\Prob}{\mathscr P}
\RenewDocumentCommand{\Pr}{d<>m}{\Prob\IfValueTF{#1}{#1}{\!\left}(#2\IfValueTF{#1}{#1}{\right})}
\NewDocumentCommand{\Prp}{d<>m}{\Prob\!_p\IfValueTF{#1}{#1}{\!\left}(#2\IfValueTF{#1}{#1}{\right})}
\NewDocumentCommand{\SubP}{d<>m}{\Prob^\leq\IfValueTF{#1}{#1}{\left}(#2\IfValueTF{#1}{#1}{\right})}
\NewDocumentCommand{\E}{d<>omo}{\mathbb E \IfValueTF{#2}{^{#2}}{} \IfValueTF{#1}{#1}{\left}( #3 \IfValueTF{#4}{ \IfValueTF{#1}{#1}{\middle}| #4}{} \IfValueTF{#1}{#1}{\right})}

\DeclareMathOperator{\dis}{dis}
\newcommand{\disint}[2]{\dis_{#1}^{#2}}
\NewDocumentCommand{\br}{d<>mmm}{\IfValueTF{#1}{#1}{\left}#2 #4 \IfValueTF{#1}{#1}{\right}#3}
\NewDocumentCommand{\pa}{d<>m}{\IfValueTF{#1}{#1}{\left}( #2 \IfValueTF{#1}{#1}{\right})}
\NewDocumentCommand{\push}{d<>m}{\Prob\IfValueTF{#1}{#1}{\!\left}(#2\IfValueTF{#1}{#1}{\right})}
\NewDocumentCommand{\FunP}{d<>mm}{\mathscr F \IfValueTF{#1}{#1}{\left}( #2 \rightsquigarrow #3 \IfValueTF{#1}{#1}{\right})}
\DeclareMathOperator\domain{dom}
\NewDocumentCommand{\dom}{d<>m}{\domain\! \IfValueTF{#1}{#1}{\left}( #2 \IfValueTF{#1}{#1}{\right}) }
\DeclareMathOperator{\Couplings}{Cpl}
\NewDocumentCommand{\Cpl}{d<>mm}{\Couplings \IfValueTF{#1}{#1}{\left}( #2, #3 \IfValueTF{#1}{#1}{\right}) }
\NewDocumentCommand{\set}{d<>mo}{\IfValueTF{#1}{#1}{\left}\{ #2 \IfValueT{#3}{\,\IfValueTF{#1}{#1}{\middle}|\, #3} \IfValueTF{#1}{#1}{\right}\}}
\newcommand{\noarg}{\_}
\DeclareMathOperator{\proj}{proj}
\DeclareMathOperator{\undis}{int}
\newcommand{\und}[2]{\undis_{#1}^{#2}}
\DeclareMathOperator{\bary}{bary}

\newcommand{\Nmap}{\mathcal N}
\newcommand{\FFF}[1][1]{\mathcal F_{#1}}
\newcommand{\PPP}[1][1]{\mathcal P_{#1}}
\newcommand\J[3]{\mathcal J_{#1,#2}^{#3}}
\newcommand{\Xb}{\overline X}
\renewcommand\H{\mathcal H}
\newcommand\funprod{\setbox0\hbox{$\prod$}\rlap{\hbox to \wd0{\hss$\circ$\hss}}\box0}
\newcommand{\I}{\mathcal I}
\newcommand{\K}{\mathcal K}
\renewcommand{\L}{\mathcal M}

\let\epsilon\varepsilon

\newcommand{\dpro}[1]{\mathbin{\tikz[x=0.6em,y=0.3em,baseline=-0.4em]{ \node[inner sep=0pt,anchor=center] at (0,0) {${\scriptstyle\otimes}$}; \node[anchor=north,inner sep=0pt] at (0,-1) {${\scriptstyle#1}$}; }}}
\newcommand{\mcmo}[1]{\mcmp_{#1}}
\newcommand{\marg}[2]{#1_{\restriction #2}}
\newcommand{\mprod}{\otimes}

\def\P{{\mathbb P}}

\def\R{{\mathbb R}}
\newcommand{\Law}{\mathscr L}

\newcommand{\Extw}{\mathcal E}

\newcommand{\Lp}{L_p}

\NewDocumentCommand{\Lpnorm}{d<>m}{\IfValueTF{#1}{#1}{\left}\lVert#2\IfValueTF{#1}{#1}{\right}\rVert_{\Lp}}

\newcommand{\cpl}{\Couplings}
\newcommand{\cpla}{\Couplings_c}
\newcommand{\cplba}{\Couplings_{bc}}

\newcommand{\WA}{\mathcal W}
\newcommand{\AWA}{\mathcal {AW}}
\newcommand{\CWA}{\mathcal {CW}}
\newcommand{\SWA}{\mathcal {SCW}}
\newcommand{\SCWA}{\mathcal{SCW}}
\newcommand{\IWA}{\mathcal{IW}}
\newcommand{\ND}{\mathcal{ND}}

\newcommand{\Wp}{\mathcal W_p}
\newcommand{\AWp}{\mathcal A \mathcal W_p}
\newcommand{\CWp}{\mathcal C \mathcal W_p}
\newcommand{\SCWp}{\mathcal S \mathcal C \mathcal W_p}

\renewcommand{\S}{\mathcal}
\newcommand{\X}{\S X}
\newcommand{\Y}{\S Y}
\newcommand{\Z}{\S Z}
\newcommand{\A}{\S A}
\newcommand{\B}{\S B}
\newcommand{\C}{\S C}
\newcommand{\Xs}{\S X}
\newcommand{\Ys}{\S Y}
\newcommand{\Zs}{\S Z}
\newcommand{\As}{\S A}
\newcommand{\Bs}{\S B}
\newcommand{\Cs}{\S C}

\newcommand{\D}{\rho}

\def\sbcorr#1#2{\def\tmpa{#1}\def\tmpb{#2}\futurelet\next\sbcorrA}
\def\sbcorrA{\ifx\next_\expandafter\sbcorrB\fi}
\def\sbcorrB_#1{_{\mkern\tmpa#1}\futurelet\next\sbcorrC}
\def\sbcorrC{\ifx\next^\expandafter\sbcorrD\fi}
\def\sbcorrD^#1{^{\mkern\tmpb#1}}

\newcommand{\overbar}[1]{{}\mkern 3.5mu\overline{\mkern-3.5mu#1\mkern-0.5mu}\mkern 0.5mu\sbcorr{-3mu}{0mu}}
\renewcommand{\Xb}{\overbar{\Xs}}
\newcommand{\Yb}  {\overbar{\Ys}}
\newcommand{\Zb}  {\overbar{\Zs}}

\newcommand{\identity}[1]{#1}
\newcommand{\dosmash}{\let\optsmash\smash}
\newcommand{\dontsmash}{\let\optsmash\identity}
\dosmash
\newcommand{\XXs}{\optsmash{\vec \Xs}}
\newcommand{\YYs}{\optsmash{\vec \Ys}}
\newcommand{\ZZs}{\optsmash{\vec \Zs}}

\newcommand{\RV}{}
\newcommand{\Xr}{\RV X}
\newcommand{\Yr}{\RV Y}
\newcommand{\Zr}{\RV Z}
\newcommand{\Ar}{\RV A}

\newcommand{\CCr}{\RV \smash{\hat C}}
\newcommand{\XXr}{\RV {\smash{\hat X_2}}}
\newcommand{\YYr}{\RV {\smash{\hat Y_2}}}
\newcommand{\ZZr}{\RV {\smash{\hat Z_2}}}
\newcommand{\VYr}{\RV \smash{\vec Y}}
\newcommand{\VZr}{\RV \smash{\vec Z}}
\newcommand{\VXXr}{\RV {\smash{\hat{\vec X}}}}

\newcommand{\overb}[3]{{}\mkern#1mu\overline{\mkern-#1mu#3\mkern-#2mu}\mkern#2mu}

\newcommand{\Wcau}{\mathcal C \mathcal W_p}
\NewDocumentCommand{\Wc}{d<>mm}{\Wcau \IfValueTF{#1}{#1}{\left}( #2, #3 \IfValueTF{#1}{#1}{\right}) }
\newcommand{\WWcau}{\overb{5}{6}{\mathcal C \mathcal W_p}}
\NewDocumentCommand{\WWc}{d<>mm}{\WWcau \IfValueTF{#1}{#1}{\left}( #2, #3 \IfValueTF{#1}{#1}{\right}) }
\newcommand{\Wcac}{\mathcal S \mathcal C \mathcal W_p}
\NewDocumentCommand{\Wca}{d<>mm}{\Wcac \IfValueTF{#1}{#1}{\left}( #2, #3 \IfValueTF{#1}{#1}{\right}) }
\newcommand{\WWcac}{\overb{5}{6}{\mathcal S \mathcal C \mathcal W_p}}
\NewDocumentCommand{\WWca}{d<>mm}{\WWcac \IfValueTF{#1}{#1}{\left}( #2, #3 \IfValueTF{#1}{#1}{\right}) }
\newcommand{\Wbc}{\mathcal A \mathcal W_p}
\NewDocumentCommand{\Wb}{d<>mm}{\Wbc \IfValueTF{#1}{#1}{\left}( #2, #3 \IfValueTF{#1}{#1}{\right}) }
\newcommand{\WWbc}{\overb{8}{7}{\mathcal A \mathcal W_p}}
\NewDocumentCommand{\WWb}{d<>mm}{\WWbc \IfValueTF{#1}{#1}{\left}( #2, #3 \IfValueTF{#1}{#1}{\right}) }
\newcommand{\Wa}{\mathcal W_p}
\NewDocumentCommand{\W}{d<>mm}{\Wa \IfValueTF{#1}{#1}{\left}( #2, #3 \IfValueTF{#1}{#1}{\right}) }
\newcommand{\Wbb}{\Wbc}
\newcommand{\Wcaa}{\Wcac'}

\newcommand{\diso}[1]{\dis_{#1}}
\newcommand{\undo}[1]{\undis_{#1}}
\newcommand{\tsint}{{\textstyle \int}}
\newcommand{\tsiint}{{\textstyle \iint}}
\newcommand{\ExpS}{\mathbb E}
\newcommand{\ellp}{\ell_p}
\usepackage{bm}
\newcommand{\funcomma}{\bm,}
\newcommand{\XP}[1]{\X_{#1:N}}
\newcommand{\XF}[1]{\X^{\mathscr F}_{#1:N}}
\newcommand{\DP}[1]{\D_{#1:N}}

\NewDocumentCommand{\CLaw}{d<>omo}{\Law \IfValueTF{#2}{^{#2}}{} \IfValueTF{#1}{#1}{\left}( #3 \IfValueTF{#4}{ \IfValueTF{#1}{#1}{\middle}| #4}{} \IfValueTF{#1}{#1}{\right})}%
\newcommand{\Xv}[2]{\bar X_{#1}^{#2}}%
\newcommand{\gobble}[1]{}
\gobble{
to execute: yank and then :@"
:let @n="dd'nP''"
:iabbrev bega \begin{align*}<CR>]]<Esc>k$a
:iabbrev begt \begin{tikzpicture}<CR>]]<Esc>k$a
:iabbrev begp \begin{proof}<CR>]]<Esc>k$a
:iabbrev begl \begin{lemma}<CR>]]<Esc>k$a
:iabbrev beg {<Esc>xxa}<Esc>hhbi\begin{}<Esc>xo]]<Esc>k$a

$\end{lemma}]\end{proof}\end{tikzpicture}\end{align*}
}

\author[J.\ Backhoff-Veraguas]{Julio Backhoff-Veraguas}
\author[D.\ Bartl]{Daniel Bartl}
\author[M.\ Beiglb\"ock]{Mathias Beiglb\"ock}
\author[M.\ Eder]{Manu Eder}
  \address{Department of Mathematics, University of Vienna, Austria}
  \email[J.\ Backhoff-Veraguas]{julio.backhoff@univie.ac.at}
  \email[D.\ Bartl]{daniel.bartl@univie.ac.at}
  \email[M.\ Beiglb\"ock]{mathias.beiglboeck@univie.ac.at}
  \email[M.\ Eder]{manuel.eder@univie.ac.at}
\date{\today}

\thanks{
J.\ Backhoff-Veraguas gratefully acknowledges financial support from the Austrian Science Fund (FWF) under grant P30750.
D.\ Bartl has been funded by the Vienna Science and Technology Fund (WWTF) through projects VRG17-005 and MA16-021, as well as by the Austrian Science Fund (FWF) through project P28661.
M.\ Beiglboeck gratefully acknowledges financial support by the FWF through grant Y782.
M.\ Eder gratefully acknowledges financial support by the FWF through grant Y782 and by the WWTF through project MA16-021.
}

\title[All Topologies are Equal]{All adapted topologies are equal}

\begin{document}

\begin{abstract}
A number of researchers have introduced topological structures on the set of laws of stochastic processes.  A unifying goal of these authors is to strengthen the usual weak topology in order to adequately capture the temporal structure of stochastic processes.

Aldous defines an extended weak topology based on the weak convergence of prediction processes.
In the economic literature, Hellwig introduced the information topology to study the stability of equilibrium problems.
Bion-Nadal and Talay introduce a version of the Wasserstein distance between the laws of diffusion processes.
Pflug and Pichler consider the nested distance (and the weak nested topology) to obtain continuity of stochastic multistage programming problems.
These distances can be seen as a symmetrization of Lassalle's causal transport problem, but there are also further natural ways to derive a topology from causal transport.

Our main result is that all of these seemingly independent approaches define \emph{the same topology} in finite discrete time. Moreover we show that this `weak adapted topology' is characterized as the coarsest  topology that guarantees continuity of optimal stopping problems for continuous bounded reward functions.
\end{abstract}


\maketitle

{\bf Keywords:} 
Aldous' extended weak topology, Hellwig's information topology, nested distance, causal optimal transport, stability of optimal stopping, Vershik's iterated Kantorovich distance


\section{Introduction}
\subsection{Outline}
If some type of natural phenomenon is modelled through a stochastic process, one might expect that the model does not describe reality in an entirely accurate way.
To be able to study the impact of such inaccuracies on the problems one is trying to solve, it makes sense to equip the set of laws of stochastic processes with a suitable notion of distance or topology.

Denoting by $\Omega := \X^N$ the path space (where $X$ is some Polish space and $N \in \N$), the set of laws of stochastic processes is $\Pr \Omega$, i.e.\ the set of probability measures on $\Omega$.

Clearly, $\Pr \Omega$ carries the usual weak topology.
 However, this topology does not respect the time evolution of stochastic processes which has a number of potentially inconvenient consequences: e.g., problems of optimal stopping / utility maximization / stochastic programming are not continuous, arbitrary processes can be approximated by processes which are deterministic after the first period, etc.
 In the following we describe a number of approaches which have been developed by different authors to deal with these (and related) problems.
 Our main result (Theorem \ref{thm:main_simple}) is that all of these approaches actually define the same topology in the present discrete time setup. Moreover, this topology is the weakest topology which allows for continuity of optimal stopping problems.

\subsection{Adapted Wasserstein distances, nested distance}\label{sec:AWIntro}
A number of authors have independently  introduced variants of the Wasserstein distance which take the temporal structure of processes into account: the definition of `iterated Kantorovich distance' by Vershik \cite{Ve70, Ve94} might be seen as a first construction in this direction.
The topic is also considered by R\"uschendorf \cite{Ru85}.
Independently, Pflug and Pflug--Pichler \cite{PfPi12, Pi13,  PfPi14, PfPi15, PfPi16, GlPfPi17} introduce the nested distance and describe the concept's rich potential for the approximation of stochastic multi-period optimization problems.
Lassalle \cite{Las18} considers the `causal transport problem' that leads to a corresponding notion of distance.
Once again independently of these developments, Bion-Nadal  and Talay \cite{BiTa19} define an adapted version of the Wasserstein distance between laws of solutions to SDEs. 
Gigli \cite[Chapter 4]{Gi04} introduces a similar distance for measures whose first marginal agrees, see also \cite[Section 12.4]{AmGiSa08}.

To set the stage for describing these `adapted' variants let us fix $p \geq 1$ and recall the definition of the usual $p$-Wassterstein distance.

\let\fancypush\push
\newcommand{\simplepush}[1]{#1_{\#}}
\let\push\simplepush
$(\X,\D_\X)$ is now a Polish metric space. On $\Omega = \X^N$ we use the Polish metric $\D_\Omega((x_t)_t, (y_t)_t) := (\sum_t \D_\X(x_t,y_t)^p)^{1/p}$.
Typically, when clear from the context we will omit the subscript for the metric.
We use $(X_t)_t$ to denote the canonical process on $\Omega$, i.e.\ $X_t$ is the projection onto the $t$-th factor of $\Omega = \X^N$.
On $\Omega \times \Omega$ call $X = (X_t)_t$ the projection on the first factor and call $Y = (Y_t)_t$ the projection on the second factor. 
For $\mu, \nu \in \Pr \Omega$ we denote by $\Cpl \mu \nu$ the set of probability measures $\pi$ on $\Omega \times \Omega$ for which $X \sim \mu$ and $Y \sim \nu$ under $\pi$, i.e.\ for which the distribution of $X$ under $\pi$ is $\mu$ and that of $Y$ under $\pi$ is $\nu$.
In applications, a particular role is played by Monge couplings.
A Monge coupling from $\mu$ to $\nu$ is a coupling $\pi$ for which $Y = T(X)$ $\pi$-a.s.\ for some Borel mapping $T:\Omega\to \Omega$ that transports $\mu$ to $\nu$, i.e.\ satisfies $\push T (\mu)=\nu$.

For $\mu, \nu \in \Prp \Omega$, i.e.\ for probability measures on $\Omega$ with finite $p$-th moment their $p$-Wasserstein distance is
\begin{align}\label{eq:UsualW}
\Wp(\mu, \nu):=\inf\left\{\E[\pi]{\D(X,Y)^p}^{1/p}: \pi \in \cpl(\mu, \nu)\right\} \fullstop
\end{align}
Following, \cite{Pr07b} the infimum in \eqref{eq:UsualW} remains unchanged  if one minimizes only over Monge couplings in many situations.

To motivate the formal definition of the adapted cousins in \eqref{eq:SCdW} and \eqref{eq:AdW}  below, we start with an informal discussion in terms of Monge mappings: 
In probabilistic terms,  the preservation of mass assumption  $\push T (\mu)= \nu$ asserts  
\begin{align}\label{MassPreservation} 
\big(T_1(X_1, \ldots, X_N), \ldots, T_N(X_1, \ldots, X_N)\big)   \sim  \nu,
\end{align}
which ignores the evolution of $\mu$ and $\nu$ (resp.) in time. Rather it would appear more natural to restrict to  mappings $(T_k)_{k=1}^N$ which are adapted in the sense that $T_k$ depends only on $X_1, \ldots, X_k$.
Adapted Wasserstein distances can be defined  following precisely this intuition, relying on a suitable version of adaptedness on the level of couplings:

The set $\cpla(\mu, \nu)$ of \emph{causal couplings}
  \footnote{Intuitively, at time $t$, given the past $(X_1, \ldots, X_t) $ of $X$,
    the distribution of $Y_t$ does not depend on the future $(X_{t+1}, \ldots, X_N)$ of $X$.
    For measures $\mu$ such that the first marginal of $\mu$ has no atoms, the weak closure of the set of adapted Monge couplings,
    i.e.\ of those $\pi \in \Cpl \mu \nu$ for which $Y=T(X)$ $\pi$-a.s.\ for $T$ adapted,
    is precisely the set of all causal couplings, see \cite{La18}.}
  consists of all $\pi\in \cpl(\mu, \nu)$ such that
\begin{align}\label{def:adapted.coupling} \pi\big( (Y_1,\dots,Y_t)\in A | X\big)& = \pi\big( (Y_1,\dots,Y_t)\in A | X_1,\dots X_t \big). 
\end{align}
for all $t\leq N$ and $A \subseteq \X^t$ measurable, cf.\ \cite{Las18}.
The set of all \emph{bi-causal} couplings $\cplba(\mu, \nu)$ consists of all $\pi\in\cpla(\mu, \nu)$ such that the distribution of $(Y,X)$ under $\pi$ is also in $\cpla(\nu,\mu)$, i.e.\ that \eqref{def:adapted.coupling} also holds with the roles of $X$ and $Y$ reversed.

The term \emph{causal} was introduced by Lassalle \cite{Las18}, who considers a causal transport problem in which the usual set of couplings is replaced by the set of causal couplings. The resulting concept is not actually a metric as it lacks symmetry, but as suggested by Soumik Pal, this is easily mended and we formally define the \emph{causal -} and \emph{symmetrized-causal $p$-Wasserstein distance}, resp.\ as follows:

For  $\mu, \nu \in \Prp \Omega$ set
\begin{align}
  \label{eq:CdW} 
    \CWA_p(\mu, \nu)  & := \inf\left\{\E[\pi]{\D(X,Y)^p}^{1/p}: \pi \in \cpla(\mu, \nu)\right\} \\
  \label{eq:SCdW}
    \SCWA_p(\mu, \nu) & := \max\pa<\big>{\CWA_p(\mu, \nu) ,\, \CWA_p(\nu, \mu)} \fullstop
\intertext{We use the term \emph{adapted Wasserstein distance} for }
  \label{eq:AdW} 
    \AWA_p(\mu, \nu)  & := \inf\left\{\E[\pi]{\D(X,Y)^p}^{1/p}: \pi \in \cplba(\mu, \nu)\right\} \fullstop
\end{align} 
  R\"uschendorf \cite{Ru85} refers to $\AWA_p$ as `modified Wasserstein distance'. 
  Pflug-Pichler \cite[Definition 1]{PfPi12} use the names multi-stage distance of order $p$ and nested distance.
  It can also be considered as a discrete time version of the `Wasserstein-type distance' of Bion-Nadal and Talay \cite{BiTa19}.
  In \cite{BaBaBeEd19a} we use a slightly modified definition of $\AWA_p$ which scales better with the number of time-periods $N$ but leads to an equivalent metric (for fixed $p$ and $N$).
  We shall discuss further properties of $\AWA_p$ (and in particular the connection with Vershik's iterated Kantorovich distance) in Section \ref{sec:DerErhabeneRaum} below.

\subsection{Hellwig's information topology}
\label{sec:HellwigIntro}

The information topology  introduced by Hellwig in \cite{He96} (as well as Aldous' extended weak topology which we discuss next) is based on the idea that an essential part of the structure of a process is the information that we may deduce about the future behaviour of the process given its behaviour up to current time $t$. For a process whose law is $\mu$, this information is captured by the conditional law $\Law^\mu(X_{t+1}, \dots, X_N | X_1=x_1, \dots, X_t=x_t) $ of $X_{t+1}, \dots, X_N$ given $X_1=x_1, \dots, X_t=x_t$ under $\mu$.

$\Law^\mu(X_{t+1}, \dots, X_N | X_1=x_1, \dots, X_t=x_t) $ is also the disintegration $\mu_{x_1,\dots,x_t}$ of $\mu \in \Pr \Omega$ w.r.t.\ the first $t$ coordinates.

Hellwig's \emph{information topology} is the initial topology w.r.t.\ a family of maps $(\I_t)_{t=1}^{N-1}$ which are defined based on these disintegrations:
\begin{align*}
  \I_t & : \Pr \Omega \rightarrow \Pr { \mathcal X ^ t \times \Pr {\mathcal X ^ {N-t}} } \\
  \I_t(\mu) & := \push {k^t}(\mu) \\
  k^t(x_1,\dots,x_N) & := (x_1,\dots,x_t,\mu_{x_1,\dots,x_t})
\end{align*}

Equivalently, $\I_t(\mu)$ is the joint law of 
\begin{align*}
  X_1, \dots, X_t, \Law^\mu(X_{t+1}, \dots, X_N | X_1, \dots, X_t)
\end{align*}
under $\mu$, and Hellwig's information topology is therefore the coarsest topology which makes continuous for all $t$ the maps which send a probability $\mu$ to the joint law describing the evolution of the coordinate process up to time $t$ and the prediction about the future behaviour of the coordinate process after $t$. 

{\color{black}
\begin{remark}
All the topologies we consider in this paper are second countable. As such they can be characterized by saying which sequences converge. Restated in the language of sequences, the above definition says that a sequence $(\mu_n)_n$ in $\mathcal{P}(\Omega)$ converges in Hellwig's information topology to $\mu\in\mathcal{P}(\Omega)$ if and only if, for every $t$, the sequence $(\I_t(\mu_n))_n$ converges to $\I_t(\mu)$ in the usual weak topology on $\Pr { \mathcal X ^ t \times \Pr {\mathcal X ^ {N-t}} }$.
\end{remark}
}

The work of Hellwig \cite{He96} was motivated by questions of stability in dynamic economic models/games; see the related articles \cite{Jo77, VZ02,HeSch02,BarbieGupta}.

\subsection{Aldous' extended weak topology}
Aldous \cite{Al81} introduces a type of convergence for pairs of filtrations and continuous time stochastic processes on them that he calls \emph{extended weak convergence} \cite[Definition 15.2]{Al81}. Restricted to our current setting, his definition can be paraphrased in a similar manner as that of the information topology.
Aldous' idea is to represent a stochastic process with law $\mu$ through the associated \emph{prediction process}\footnote{The  definition of the prediction process goes back at least to Knight \cite{Kn75}.}, that is, the process given by
\begin{align*}
  Z^\mu_0:= \Law(X) = \mu, Z^\mu_1:= \Law^\mu(X | X_1),  \dots, Z^\mu_N:=\Law^\mu(X | X_1, \dots, X_N).
\end{align*}
That is, $(Z^\mu_t)_{t=0}^N$ is a measure-valued martingale that makes increasingly accurate predictions about the full trajectory of the process $X$. 

Rather then comparing  the laws of processes directly, the \emph{extended weak topology} is  derived from the weak topology on the corresponding prediction processes (plus the original processes).
I.e.\ formally, the extended weak topology on $\Pr{\Omega}$ is the initial topology w.r.t.\ the map
\begin{align*}
  \Extw : \Pr \Omega \rightarrow \Pr { \Omega \times \Pr \Omega ^ {N+1} }
\end{align*}
which sends $\mu$ to the joint distribution of
\begin{align*}
   (X, Z^\mu) =\left(  X_1, \dots, X_N, \mu, \Law^\mu(X | X_1), \Law^\mu(X | X_1, X_2), \dots, \Law^\mu(X | X_1, \dots, X_N)\right)
\end{align*}
under $\mu$.

Note that, to stay faithful to Aldous' original definition, we defined $\Extw$ to map $\mu$ not just to the law of the prediction process but to the joint law of the original process and its prediction process. One easily checks that the original process may be omitted in our setting without changing the resulting topology.

\subsection{The optimal stopping topology}

The usual weak topology on $\Pr \Omega$ is the coarsest topology which makes continuous all the functions
\begin{align*}
  \mu \mapsto \tsint f \d\mu
\end{align*}
for $f: \Omega \rightarrow \R$ continuous and bounded.

One may follow a similar pattern and look at the coarsest topology which makes continuous the outcomes of all sequential decision procedures.
Perhaps the easiest way to formalize this is to look at optimal stopping problems.
In detail, 
write $AC(\Omega)$ for the set of all processes $(L_t)_{t=1}^N$ which are adapted, bounded and satisfy that $x\mapsto L_t(x) $ is continuous for each $t\leq N$. 
Write $v^L(\mu)$ for the corresponding value function, given that the process $X$ follows the law $\mu$, i.e.\
\begin{align*}
  v^L(\mu):=\inf\{\E[\mu] {L_\tau}: \tau\leq N\mbox{ is a stopping time}\}.
\end{align*}

The \emph{optimal stopping topology} on $\Pr \Omega$ is the coarsest topology which makes the functions
\begin{align*}
  \mu \mapsto v^L(\mu)
\end{align*}
continuous for all $(L_t)_{t=1}^N \in AC(\Omega)$.

\subsection{Main result}

We can now state our main result:

\begin{theorem}
  \label{thm:main_simple}
  Let $(\X,\D_\X)$ be a Polish metric space, where $\D_\X$ is a bounded metric and set $\Omega := \X^N$. Then the following topologies on $\Pr \Omega$ are equal
  \begin{enumerate}
    \item\label{it:AW}
      the topology induced by $\AWp$
    \item\label{it:SCW}
      the topology induced by $\SCWp$
    \item\label{it:Hellwig}
      Hellwig's information topology
    \item\label{it:Aldous}
      Aldous' extended weak topology
    \item\label{it:optstop}
      the optimal stopping topology.
  \end{enumerate}
\end{theorem}

The assumption that $\D_\X$ is bounded serves only to simplify the statement of the theorem, because in this case the topology induced by $\Wp$ coincides with the weak topology. For every Polish space there is a bounded complete metric which induces the topology (given any complete metric $\D_\X$, replace it by e.g.\ $\min(1,\D_\X)$).

\subsubsection{$p$-Wasserstein and unbounded metrics}
There is an analogous statement, Theorem \ref{thm:main} below, which drops the assumption that $\D_\X$ is bounded.
To be able to state it, we introduce slight variations of Hellwig's information topology, of Aldous' extended weak topology and of the optimal stopping topology:

In \cite{He96} Hellwig equips the target spaces of $\I_t$ with the weak topology -- or more precisely he equips $\Pr {\X^{N-t}}$ with the weak topology, $\X^t \times \Pr {\X^{N-t}}$ with the product topology and finally $\Pr {\X^t \times \Pr {\X^{N-t}}}$ with the weak topology based on this topology.
One may easily define a $p$-Wasserstein version of Hellwigs information topology by using the recipe \enquote*{replace the weak topology by the $p$-Wasserstein metric everywhere}. Concretely, if we restrict $\I_t$ to $\Prp \Omega$, we may view it as a map into $\Prp { {\X^t \times \Prp {\X^{N-t}}} }$, where the last space carries the metric
\begin{align*}
  \D_{\Prp { {\X^t \times \Prp {\X^{N-t}}} }}(\mu,\nu) & := \inf_{\gamma \in \Cpl \mu \nu} \pa<\Big>{ \tsint \D((x_i)_{i\leq t}, (y_i)_{i \leq t})^p \\
  & \phantom{:=}\quad + \Wp (\hat \mu, \hat \nu)^p \, \d\gamma((x_i)_{i\leq t},\hat \mu, (y_i)_{i\leq t}, \hat \nu) }^{1/p} \fullstop
\end{align*}
We will call the resulting variant of Hellwigs information topology on $\Prp \Omega$ the \emph{$\Wp$-information topology}.

Similarly, one may systematically replace every occurrence of the weak topology in the definition of the extended weak topology by the $p$-Wasserstein metric. We call the resulting topology on $\Prp \Omega$ the \emph{extended $\Wp$-topology}.

Just like the weak topology is the coarsest topology which makes integration of continuous bounded functions continuous, the $p$-Wasserstein topology is the coarsest topology which makes integration of continuous functions bounded by $c \cdot (1 + \D(x_0,x)^p)$ continuous.
Following this analogy, we define $AC_p(\Omega)$ as the set of all processes $(L_t)_{t=1}^N$ which are adapted, bounded by $x \mapsto c \cdot (1 + \D(x_0,x)^p)$ for some $c \in \R_+$ and satisfy that $x\mapsto L_t(x) $ is continuous for each $t\leq N$.

  The \emph{$\Wp$-optimal stopping topology} on $\Prp \Omega$ is the coarsest topology which makes the functions
  \begin{align*}
    \mu \mapsto v^L(\mu)
  \end{align*}
  continuous for all $(L_t)_{t=1}^N \in AC_p(\Omega)$.

With these we may state the following generalization of Theorem \ref{thm:main_simple}:

\begin{theorem}
  \label{thm:main}
  Let $(\X,\D_\X)$ be a Polish metric space and set $\Omega := \X^N$. Then the following topologies on $\Prp \Omega$ are equal
  \begin{enumerate}
    \item the topology induced by $\AWp$
    \item the topology induced by $\SCWp$
    \item the $\Wp$-information topology
    \item the extended $\Wp$-topology
    \item the $\Wp$-optimal stopping topology.
  \end{enumerate}
\end{theorem}

Clearly, one recovers Theorem \ref{thm:main_simple} from Theorem \ref{thm:main} by choosing a bounded metric on $\X$, because the $\Wp$-information topology for bounded $\D_\X$ is just the information topology, the extended $\Wp$-topology for bounded $\D_\X$ is just the extended weak topology and the $\Wp$-optimal stopping topology for bounded $\D_\X$ is just the optimal stopping topology.

The relationship between the topologies listed in Theorem \ref{thm:main_simple} and those listed in Theorem \ref{thm:main} is similar to the non-adapted case where we know that usual $p$-Wasserstein convergence is equivalent to usual weak convergence plus convergence of the $p$-th moments.

\begin{restatable}{lemma}{convergencemoments}
\label{lem:convergence.moments}
  Convergence in any of the topologies of Theorem \ref{thm:main} is equivalent to convergence in any of the topologies of Theorem \ref{thm:main_simple} (where for building $\SCWp$ and $\AWp$, $\D_\X$ is replaced by a bounded compatible complete metric e.g.\ $\min(1,\D_\X)$) plus convergence of $p$-th moments on $\Omega$ w.r.t.\ (the original) $\D_\Omega$.
\end{restatable}
We prove Lemma \ref{lem:convergence.moments} in Section \ref{sec:convergence.moments}, making use of (parts of) Theorem \ref{thm:main_simple} and Theorem \ref{thm:main}.

\subsection{Further  remarks on related work} 
\subsubsection{Some further articles of successors of Aldous}
 One of the original applications of Aldous' weak extended topology concerned the stability of optimal stopping \cite{Al81}. This corresponds to one half of \ref{it:Aldous}=\ref{it:optstop} in Theorem \ref{thm:main_simple}, but in a much more general setting. This line of work has been continued by Lamberton and Pag\`es \cite{LaPa90}, Coquet and Toldo \cite{CoTo07}, among others.

Aldous' extended weak topology was also inspiring and instrumental for the development of the theory of convergence of filtrations, and the associated questions of stability of the martingale representation property and Doob-Meyer decompositions. In this regard, see the works by Hoover et al \cite{HoKe84,Ho91} and by M\'emin et al \cite{CoMeSl01,Me03}. The related question of stability of stochastic differential equations (as well as their backwards version) with respect to the \textit{driving noise} has particularly seen a burst of activity in the last two decades. For brevity's sake we only refer to the recent article by Papapantoleon, Posama\"i, and Saplaouras \cite{PaPoSa18} for an overview of the many available works in this direction.

\subsubsection{Previous applications of adapted Wasserstein distances.}
 
Pflug, Pichler and co-authors \cite{PfPi12, Pi13,  PfPi14, PfPi15, PfPi16, GlPfPi17} have extensively developed and applied the notion of nested distaces for the purpose of scenario generation, stability, sensitivity bounds, and distributionally robust stochastic optimization, in the context of operations research.

Acciaio, Zalashko, and one of the present authors consider in \cite{AcBaZa16} the adapted Wasserstein distance in continuous time in connection with utility maximization, enlargement of filtrations and optimal stopping. 

Causal couplings have appeared in the work by Yamada and Watanabe \cite{YW}, Jacod and M\'emin \cite{JaMe81} as well as Kurtz \cite{Ku07,Ku14}, concerning weak solutions of stochastic differential equations, and by R\"uschendof \cite{Ru85} concerning approximation theorems in probability theory. The term `causal' is first used by Lassalle \cite{Las18}, who uses it in an additional constraint for the transport problem and gives an alternative derivation of the Talagrand inequality for the Wiener measure. 
 Causal couplings are also present in the numerical scheme suggested in \cite{AcBaCa18} for (extended mean-field) stochastic control.  
 
The article \cite{BaBeHuKa17} connects  adapted Wasserstein distance (in continuous time) to martingale optimal transport  (cf.\ \cite{HoNe12,  BeHePe12, GaHeTo13, DoSo12, BoNu13, HoKl13, CaLaMa14, BeCoHu14, BeNuTo16} among many others). Several familiar objects appear as solutions to variational problems in this context. E.g.\ geometric Brownian motion is the martingale which is closest in $\AWA_2$ to usual Brownian motion subject having a log normal distribution at the terminal time-point, the local vol model is closest to Brownian motion subject to matching 1-d marginals.  
 
Bion-Nadal and Talay \cite{BiTa19} introduce an adapted Wasserstein-type distance on the set of diffusion SDEs and show that this distance corresponds to the computation of a tractable stochastic control problem. They also apply their results to the problem of fitting diffusion models to given marginals. 
 
In \cite{BaBaBeEd19a} the present authors consider adapted Wasserstein distances in relation to stability in finance:
Lipschitz continuity of utility maximization/hedging are established w.r.t.\ to the underlying models in discrete and continuous time.

\subsection{Another formulation of the adapted Wasserstein distance and of Hellwigs information topology}
\label{sec:DerErhabeneRaum}

Here we give an alternative formulation of the adapted Wasserstein distance / nested distance due to  Pflug and Pichler.

Again, $\X$ is a Polish space and $\D = \D_\X$ is a compatible metric on $\X$. Starting with $V_N^p:=0$ we define
\begin{align}\label{valuefunctiongeneral}
	&V_t^p(x_1,\dots,x_{t},y_1,\dots,y_{t}):= \\
	&\inf_{\gamma^{t+1} \in \cpl(\mu_{x_1,\dots,x_{t}}, \nu_{y_1,\dots,y_{t}} )} \iint \! \left( \! \begin{array}{cc}V^p_{t+1}(x_1,\dots,x_{t+1},y_1,\dots,y_{t+1})\\+\ \D(x_{t+1},y_{t+1})^p\end{array} \! \right) \! \d\gamma^{t+1}(x_{t+1}, y_{t+1}).\nonumber
\end{align}
 The nested distance is finally obtained in a backwards recursive way by
\begin{equation}\label{eq:Nested}
  \ND_p(\mu,\nu)^p =\inf_{ \gamma^1 \in \cpl(\push {{\proj_1}} (\mu),\push {{\proj_1}} (\nu)) }
	\iint	\left( 	V^p_1(x_1,y_1)+\ \D(x_1,y_1)^p
		\right)\d\gamma^1(x_1,y_1).
\end{equation}
Then $\AWp = \ND_p$. We refer to \cite{BaBeLiZa16} for the (straightforward) justification.

For $N>1$ the adapted Wasserstein distance is not complete. As was established in \cite{BaBeEdPi17}, a natural complete space into which $(\Prp{\Omega},\AWA_p)$ embeds is given by the space of \emph{nested distributions}:

Consider the sequence of metric spaces
\begin{align*}
	\XP{N}   &:= (\mathcal{X}, \DP{N}), 
  & \DP{N}   & :=\D=(\D^p)^{1/p},\\
	\XP{N-1} &:= \big(\mathcal{X}\times \Prp{\XP{N}},  \DP{N-1}\big),
  & \DP{N-1} & :=\left (\D^p + \WA_{\DP{N},p}^p\right)^{1/p},\\
  & \vdots & \vdots \nonumber\\
	\XP{1}   &:= \big(\X \times \Prp{\XP{2}}, \DP{1}\big),
  & \DP{1}    & :=\left(\D^p + \WA_{\DP{2},p}^p\right)^{1/p},
\end{align*}
where at each stage $t$, the space $\Prp{\XP{t}}$ is endowed with the $p$-Wasserstein distance with respect to the metric $\DP{t}$ on $\XP{t}$, which we denote  by $\WA_{ \DP{t},p}$. 
		The space of \emph{nested distributions} (of depth $N$) is defined as $\Prp{\XP{1}}$.
We endow $\Prp{\XP{1}}$ with the complete metric $\WA_{\DP{1},p}$.

The space of nested distributions was defined by Pflug \cite{Pf09}. Notably the idea to iterate the formation of Wasserstein spaces and metrics  goes back to Vershik \cite{Ve70, Ve94} who uses the name `iterated Kantorovich distance'. The main interest of Vershik (and his successors)  lies in the classification of filtrations (in the language of ergodic theory). We refer to the work of Emery and Schachermayer \cite{EmSc01} for a survey from a probabilistic perspective and to Janvresse, Laurent and de la Rue \cite{JaLaDe16} for a contemporary article (again from a probabilistic viewpoint). 

$\Prp{\Omega}$ is naturally embedded in the set of nested distributions of depth $N$ through the map $\Nmap$ given by
\begin{align}
  \label{eq:Nmapintro}
  &  \Nmap (\mu) := \CLaw{X_1, \CLaw{X_2, \cdots \CLaw{X_{N-1}, \CLaw{X_N}[\Xv 1 {N-1}] }[\Xv 1 {N-2}] \cdots }[X_1]}
\end{align}
where $(X_1,\dots,X_N)$ is a vector with law $\mu$, $\Law$ again denotes (conditional) law and we use $\Xv 1 t$ as a shorthand for the vector $X_1, \dots, X_t$.

Following \cite{BaBeEdPi17}, we have:
\begin{theorem}
  \label{thm:Nembedding}
  The map $\Nmap$ defined in \eqref{eq:Nmapintro} embeds the metric space $(\Prp \Omega,\AWA_p)$ \emph{isometrically} into the complete separable metric space $(\Prp{\XP 1}, \WA_{\DP{1},p})$. 
\end{theorem}

\begin{remark}
  When $\X$ has no isolated points, $\Prp{\XP 1}$ is actually the completion of $\Prp \Omega$, i.e. $\Prp \Omega$ considered as a subset of $\Prp{\XP 1}$ is dense.
\end{remark}

\subsubsection{Hellwig's information topology in terms of adapted Wasserstein distances}
\label{sec:HellwigAdap}
We note that   Hellwig's definition of the information topology can also be rephrased  using the concept of adapted Wasserstein distance:
Assume that $\D_\X$ is a bounded metric and for $t \leq N$, set 
$$\Omega = {\mathcal X}^N=\underbrace{{\mathcal X}^t}_{=:X_1^{(t)}}\times \underbrace{{\mathcal X}^{N-t}}_{=:X_2^{(t)}}=X_1^{(t)} \times X_2^{(t)}.$$ I.e.\  for each $t$, we consider $\Omega$ as the product of two Polish spaces (which one might consider as `history' and `future'). Extending the defintion of $\AWp$ in the obvious way to products of not necessarily equal Polish spaces, we can then equip $\Prp {X_1^{(t)} \times X_2^{(t)}}$ with a \emph{one period} adapted Wasserstein distance $\AWA_{p}^{(t)}, p\geq 1$. 
Setting  for $\mu, \nu\in \mathcal P(\Omega)$ 
\begin{align}
\IWA_p(\mu, \nu):= \sum_{t=1}^N \AWA_{p}^{(t)}(\mu, \nu),\quad  p\geq 1,
\end{align}
we obtain a compatible metric for the information topology. This is relatively straightforward (whereas the full version of Theorem \ref{thm:main_simple} is not straightforward as far as we are concerned).

\subsection{Preservation of Compactness}
We close this section with a result about the preservation of relative compactness which we shall use in Sections \ref{sec:CausalAndAnti} and \ref{sec:convergence.moments}, but which also might be of independent interest. Specifically, in \cite{BaBePa18,BaPa19} the two-step version of Lemma \ref{lem:CompactnessLemma} is used as a crucial tool in the investigation of the weak transport problem.

A more detailed investigation of compactness in $\Pr \Omega$ with the weak adapted topology is the topic of the companion paper to this one, \cite{ModulusOfContinuity}.

Assume for simplicity that $\D_\X$ is a bounded metric. Then we have

\begin{lemma}[Compactness lemma]\label{lem:CompactnessLemma}
  $ A \subseteq \Pr\Omega$ is relatively compact w.r.t.\ the usual weak topology iff $\Nmap[A] \subseteq \Pr{\XP 1}$ is relatively compact.
\end{lemma} 
We note that Lemma \ref{lem:CompactnessLemma} is essentially a consequence of the characterization of compact subsets in $\Pr{\Pr{X}}$; in a somewhat different framework it was first proved in \cite{Ho92}. The version stated here follows by repeated application of \cite[Lemma 3.3]{ModulusOfContinuity}/\cite[Lemma 2.6]{BaBePa18}.

The implication that $\Nmap[A]$ relatively compact implies $A$ relatively compact is rather easy to see, but the other direction that $A$ relatively compact implies $\Nmap[A]$ relatively compact is nontrivial since the mapping $\Nmap: \Pr\Omega \to\Pr{\mathcal X_{1:N}}$ is not continuous when $\Pr\Omega$ is endowed with the usual weak topology (except for trivial cases).
Lemma \ref{lem:CompactnessLemma} would \emph{not} be true if we were to replace relative compactness by compactness.

The assumption that $\D_\X$ is bounded is inessential. A version of Lemma \ref{lem:CompactnessLemma} holds if we replace $\Pr \Omega$ by $\Prp \Omega$ and the weak topology by the one induced by the $p$-Wasserstein metric.

A similar result based on Hellwig's information toplogy, relating relative compactness in $\Pr \Omega$ to relative compactness in $\prod_{t=1}^{N-1} \Pr {\X^t \times \Pr {\X^{N-t}}}$, is also true.

\let\push\fancypush

\section{Preparations}

The rest of the paper will essentially be devoted to proving Theorem \ref{thm:main_simple}, or really its generalization Theorem \ref{thm:main}.

In Section \ref{sec:arrows} we prove that Hellwig's information topology equals the topology induced by $\AWp$, i.e.\ $\ref{it:Hellwig} = \ref{it:AW}$ in Theorem \ref{thm:main}.
In a sense, of all the topologies listed in Theorem \ref{thm:main}, Hellwig's information toplogy \enquote*{looks} the coarsest -- or at least like one of the coarser ones, while the topology induced by $\AWp$ \enquote*{looks} the finest.

In Section \ref{sec:CausalAndAnti} we sandwich the topology induced by $\SCWp$ between Hellwig's information topology and the toplogy induced by $\AWp$, i.e.\ we show $\ref{it:Hellwig} \leq \ref{it:SCW} \leq \ref{it:AW}$ in Theorem \ref{thm:main}.

In Section \ref{sec:Aldous} we show that Aldous' extended weak topology is equal to Hellwig's information topology, i.e. $\ref{it:Aldous} = \ref{it:Hellwig}$ in Theorem \ref{thm:main}.

In Section \ref{sec:convergence.moments} we prove Lemma \ref{lem:convergence.moments}.

In Section \ref{sec:OptStop} we prove that the optimal stopping topology is coarser than the topology induced by $\AWp$ and finer than Hellwig's ($\Wp$-)information topology, i.e. $\ref{it:Hellwig} \leq \ref{it:optstop} \leq \ref{it:AW}$ in Theorem \ref{thm:main}.

\subsection{Notation}

The nested structure of spaces like for example $\Prp {\XP{1}}$ introduced in Section \ref{sec:DerErhabeneRaum} is (at least for the authors) not so easy to gain an intuition for. It seems rather challenging to picture probability measures on probability measures on probability measures\dots{} etc.

Therefore, much of the proofs in the following two sections will be about bookkeeping and not getting lost in these nested structures.
In most other contexts we would regard such bookkeeping as abstract nonsense better swept under the rug, but in the context of the present paper we believe that it really constitutes an important and nontrivial ingredient in successfully carrying out the proofs.

To aid in this endeavour we make some notational preparations and introduce a few conventions.

\subsubsection{Operations on Spaces}
In the introduction we described the topologies listed in Theorems \ref{thm:main_simple} and \ref{thm:main} as initial topologies w.r.t.\ maps into more complex spaces. These spaces are built up from just a few basic operations, and in most cases the maps can also be constructed using a few relatively simple ingredients.

For spaces, the operations in question are
\begin{itemize}
  \item product formation, i.e.\ for spaces $\X$ and $\Y$ we may form their product space $\X \times \Y$,
  \item and passing from a space $\X$ to the space $\Pr \X$ of probability measures on $\X$.
\end{itemize}

Here we run into some tension between the various existing definitions in the literature. While Hellwig and Aldous originally defined their topologies based on equipping the space $\Pr \X$ of probability measures on some space $\X$ with the weak topology, without any mention of metrics, $\AWp$ is a metric built on the $p$-Wasserstein metric, and Theorem \ref{thm:Nembedding} exhibits this metric as the \enquote*{initial metric} w.r.t.\ an embedding of $\Prp \Omega$ (not $\Pr \Omega$) into $(\Prp{\XP 1}, \WA_{\DP{1},p})$.

Luckily, when the base metric $\D_\X$ on $\X$ is bounded and we decide that we only care about topologies and not the metrics that induce them, all of these distinctions vanish, and one may hope for these fine distinctions to not be so important in the end.

To give as uniform and as streamlined a treatment as possible of all the various ways in which these metric and topological spaces can be related to each other we employ the following strategy:
A lot of our arguments are agnostic to the distinction between $\Prob$ and $\Prob_p$, and to whether we are talking about metric or topological spaces etc. They only rely on properties of the operations of product formation and formation of spaces of probability measures and on properties of maps between various spaces built using these operations which hold in either case.
For the rest of the paper we will therefore drop the $p$ in $\Prob_p$ and other explicit mentions of these distinctions. The reader may decide to read the paper using either of the following two sets of conventions, which are to be applied recursively:

\noindent{\bf Convention 1 (weak topologies)}
\begin{itemize}
  \item $\X$, $\Y$, $\Z$, $\A$, $\B$, $\C$, etc. are Polish spaces.
  \item $\X \times \Y$ is a topological space with the product topology (again Polish).
  \item $\Pr \X$ is a topological space with the weak topology (also Polish).
  \item \enquote*{space} will mean Polish space.
\end{itemize}

\noindent{\bf Convention 2 ($\Wp$)}
\begin{itemize}
  \item $p \geq 1$ is fixed throughout the paper
  \item $\X$, $\Y$, $\Z$, $\A$, $\B$, $\C$, etc. are Polish (i.e.\ complete separable) metric spaces with metrics $\D_\X$, $\D_\Y$, $\D_\Z$, $\D_\A$, $\D_\B$, $\D_\C$, etc.\ respectively.
  \item $\X \times \Y$ is a Polish metric space with the metric
    \begin{align*}
      \D_{\X \times \Y}((x_1,y_1),(x_2,y_2)) := \pa{\D_\X(x_1,x_2)^p + \D_\Y(y_1,y_2)^p}^{1/p} \fullstop
    \end{align*}
  \item $\Pr \X$ is a Polish metric space with the $p$-Wasserstein metric
    \begin{align*}
      \D_{\Pr \X}(\mu,\nu) := \inf_{\gamma \in \Cpl \mu \nu} \pa{ \tsint \D(x_1,x_2)^p \d\gamma(x_1,x_2) }^{1/p} \fullstop
    \end{align*}
  \item The subscript on the metric $\D$ may be dropped when clear from the context.
  \item \enquote*{space} will mean Polish metric space.
\end{itemize}

Unless specified otherwise everything said from here on will be true for either way of reading.
Convention 1 will lead to a direct proof of Theorem \ref{thm:main_simple}, while Convention 2 will give a proof of the more general version, Theorem \ref{thm:main}.
Occasionally an argument will require us to talk directly about metrics to establish continuity of some map. When one only cares about Theorem \ref{thm:main_simple} and not Theorem \ref{thm:main} these sections can be read while assuming that $p=1$ and that all metrics mentioned are bounded.

Another space we will need is
\begin{definition}
  $ \FunP \A \B \subseteq \Prob(\A \times \B) $ is the space of probability measures on $\A \times \B$ which are concentrated on the graph of a measuruable function, i.e.:
  \begin{multline*}
    \FunP \A \B := \set<\Big>{ \mu \in \Prob(\A \times \B) }[ \exists f : \A \rightarrow \B \text{ measurable s.t. } \mu(\mathrm{graph}(f)) = 1 ] \fullstop
  \end{multline*}
  The space $\FunP \A \B$ carries the subspace topology / the restriction of the metric on $\Pr {\A \times \B}$.
\end{definition}

\subsubsection{Maps between spaces}
Assuming Convention 1, when $f: \X \to \Y$ is a continuous map, the pushforward under $f$, i.e.\ the map which sends $\mu \in \Pr \X$ to the measure $\nu \in \Pr \Y$ with $\nu(A) = \mu(f^{-1}[A])$ is also continuous.

Similarly, assuming Convention 2, when $f: \X \to \Y$ is a Lipschitz-continuous map between metric spaces the pushforward under $f$ is also Lipschitz-continous from $\Pr \X$ to $\Pr \Y$.

We will use $\push f: \Pr \X \to \Pr \Y$ to denote the pushforward under $f$, to emphasize the fact that $\Prob$ is a \emph{functor}, i.e.\ that it sends a diagram with a \enquote*{nice} (read continuous/Lipschitz) map
\begin{align*}
  \X \overset{f}\longrightarrow \Y
\end{align*}
to a similar diagram
\begin{align*}
  \Pr \X \overset{\push f}\longrightarrow \Pr \Y
\end{align*}
where the map is also \enquote*{nice}, and that $\push {f \circ g} = \push f \circ \push g$ and $\push {\id\X} = \id{\Pr \X}$ (where $\id\X$ is the identity function on $\X$).

For a product of spaces $\X \times \Y$, the projection onto $\X$ will alternatively be denoted by either $\proj_\X$ or by the same letter that is used for the space, but in a non-calligrapic font, i.e. $X : \X \times \Y \to \X$.

If $\mu$ is defined on some product $\prod_i \Xs_i$ of spaces, we also introduce a shorthand notation for marginals of $\mu$, i.e.\ for the pushforward of $\mu$ under projection onto the product of some subset of the original factors:
\begin{align*}
  \marg \mu {(\Xs_{i_j})_j} = \push {(\Xr_{i_j})_j} (\mu) \fullstop
\end{align*}

If $f: \A \rightarrow \B$ and $g: \A \rightarrow \C$ are functions we write $(f \funcomma g)$ for the function
\begin{align*}
  (f \funcomma g) & : \A \rightarrow \B \times \C \\
  (f \funcomma g) (a) & := (f(a),g(a)) \fullstop
\end{align*}

If we want to specify a map from, say $\A \times \B \times \C$ to $\X$ but we only really care about one of the variables we will use an underscore \enquote*{$\noarg$} instead of naming the unused variables, as in $(a,\noarg,\noarg) \mapsto f(a)$. Similarly, when integrating we may also use $\noarg$ to denote unused variables, i.e. for $\mu \in \Pr {\X \times \Y}$ we might write $\tsint f(y) \d\mu(\noarg,y)$.

Two important maps will be the disintegration map $\disint \A \B$ and its left inverse $\und \A \B$.

The disintegration map
  $$ \disint \A \B : \Prob(\A \times \B) \rightarrow \FunP \A {\Prob(\B)} $$
sends a probability $\mu$ on $\A \times \B$ to the measure
  $$ \push<\big>{(a,\noarg) \mapsto (a,\mu_a)}(\mu) $$
where $a \mapsto \mu_a$ is a classical disintegration of $\mu$, i.e. if $\bar \mu = \disint \A \B (\mu)$ then
  $$ \int f(a,b) \d\nu(b) \d\bar\mu(a,\nu) = \int f(a, b) \d\mu_a(b) \d\mu(a,\noarg) = \int f(a,b) \d\mu(a,b) \fullstop$$
The disintegration map is measurable (see for example \cite[Proposition 7.27]{BeSh78}) and injective. It is not continuous w.r.t.\ the weak topologies or the Wasserstein metrics.

When writing $\disint {\S A} {\S B}$ we will not insist that $\S A$ has to be the first factor in the domain of $\disint {\S A} {\S B}$ -- $\S A$ and $\S B$ may even be products themselves, whose factors are intermingled in the product that makes up the domain of $\disint {\S A} {\S B}$. Also, we may sometimes omit $\S B$, only specifying the variable(s) w.r.t.\ which we are disintegrating, not the ones which are left over, as in $\diso {\S A}$.

The map
\begin{align*}
  \und \A \B & : \Pr { \A \times \Pr \B } \rightarrow \Pr { \A \times \B } \\
  \und \A \B & (\mu) := f \mapsto \int f(a,b) \d\nu(b) \d\mu(a,\nu)
\end{align*}
is (Lipschitz-)continuous.

The pair $\disint \A \B$, $\und \A \B$ enjoy the following properties:
\begin{enumerate}
  \item $\und \A \B$ is the left inverse of the disintegration map, i.e.\ 
    $$ \und \A \B \circ \disint \A \B = \id {\Pr { \A \times \B }} \fullstop $$
    This is a direct consequence of the definition of the disintegration.
  \item $\restr{\und \A \B}{\FunP \A {\Pr \B}}$ is injective. Therefore,
  \item $\disint \A \B \circ \restr{\und \A \B}{\FunP \A {\Pr \B}} = \id { \FunP \A {\Pr \B} }$, i.e.\ $\disint \A \B$ and $\und \A \B$ are inverse bijections between $\Pr { \A \times \B}$ and $ \FunP \A {\Pr \B} $.
\end{enumerate}
The last two properties are just a reformulation of the known fact that the disintegration of a measure is almost-surely uniquely defined.

\subsubsection{Processes which take values in different spaces at different times}

Already in the introduction, in Section \ref{sec:HellwigAdap}, we found it convenient to extend the definition of $\AWp$ to products of not necessarily equal Polish spaces \enquote*{in the obvious way}. To accommodate for reapplication of concepts in a similar style as seen there we make the minor generalization of letting all the processes we talk about take values in different spaces at different times -- typically at time $t$ they will take values in a space $\X_t$.

Denote by $\Xb_j^k := \prod_{i=j}^k \X_i $ and define $\Xb := \Xb_1^N$, $\Xb^k := \Xb_1^k$, $\Xb_j := \Xb_j^N$.

\section{Hellwig's \texorpdfstring{$\Wp$-}{Wp-}information topology is equal to the topology induced by \texorpdfstring{$\AWp$}{the adapted Wasserstein distance}}
\label{sec:arrows}

In this section we show $\ref{it:Hellwig} = \ref{it:AW}$ in Theorem \ref{thm:main}. We will do so by identifying both topologies as initial topologies w.r.t.\ a single map each, i.e.\ finding a space which is homeomorphic to $\Pr \Xb$ with Hellwig's ($\Wp$-)information topology and one which is homeomorphic to $\Pr \Xb$ with the topology induced by $\AWp$ and then showing that these spaces are homeomorphic in the right way. As an auxilliary tool we will introduce another topology on $\Pr \Xb$ which wasn't mentioned in the introduction, but which is very similar to Hellwig's.
The proof strategy can be summarized by saying that we want to show that the following diagram is commutative.
\begin{align}
  \label{diag:contmaps}
\begin{tikzpicture}[x=3cm,y=2cm,baseline={([yshift=-.5ex]current bounding box.center)}]
  \node (PX)   at  (0,1) {$ \Pr {\Xb} $};
  \node (FFF)  at  (1,0) {$ \FFF $};
  \node (I)    at  (0,0) {$ \I\br[]{\Pr {\Xb}} $};
  \node (Ip)   at (-1,0) {$ \I'\br[]{\Pr {\Xb}} $};
  \draw[->] (PX)  edge node[auto] {$ \Nmap $} (FFF);
  \draw[->] (PX)  edge node[auto] {$ \I $} (I);
  \draw[->] (PX)  edge node[left] {$ \I' $} (Ip);
  \draw[->,out=300,in=240] (Ip) edge node[below] {$ \H $} (FFF);
  \draw[->] (FFF) edge node[auto] {$ \K $} (I);
  \draw[->] (I)   edge node[auto] {$ \L $} (Ip);
\end{tikzpicture}
\end{align}
Here $\Nmap$ is the map which induces the same topology as $\AWp$, $\I$ induces Hellwig's topology and $\I'$ induces what we will call the \emph{reduced} information topology. We shortly restate their definitions below.

{\color{black}
Since these mappings are injective and by the definition of the initial topology all of these mappings are homeomorphisms. To be precise, $\Nmap$ is a homeomorphism from ${\Pr {\Xb}}$ with the topology induced by $\AWp$ onto $\Nmap[{\Pr {\Xb}}]$ {(cf.\ Theorem \ref{thm:Nembedding})},  $\I$ is a homeomorphism from ${\Pr {\Xb}}$ with the information topology onto $\I[{\Pr {\Xb}}]$, and 
$\I'$ is a homeomorphism from ${\Pr {\Xb}}$ with the reduced information topology  onto $\I'[{\Pr {\Xb}}]$.
}

 The maps $\K$, $\L$, $\H$ are still to be found.

As introduced in Section \ref{sec:HellwigIntro} Hellwig's ($\Wp$-)information topology is induced by a family of maps $\I_t$, given by:
\begin{align*}
  \I_t & : \Pr {\Xb} \rightarrow \FunP {\Xb^t} {\Pr{ \Xb_{t+1} }} \\
  \I_t & := \disint {\Xb^t} {\Xb_{t+1}} \fullstop
\end{align*}
Equivalently, the information topology is the initial topology w.r.t.\ the map
\begin{align*}
  \I & : \Pr {\Xb} \rightarrow \prod_{t=1}^{N-1} \FunP {\Xb^t} {\Pr{ \Xb_{t+1} }} \\
  \I(\mu) & := (\I_t(\mu))_t \fullstop
\end{align*}
We saw in Section \ref{sec:DerErhabeneRaum} that $\AWp$ is induced by an embedding $\Nmap: \Pr \Xb \to \Pr {\XP 1}$. Rephrasing the definition there, $\Nmap$ is obtained by defining recursively from $t=N-1$ to $t=1$:
\begin{align*}
  \Nmap^{N}   & := \id{\Pr \Xb} \\
  \Nmap^{t}   & := \disint {\Xb^t} {\XP{t+1}} \circ \Nmap^{t+1}
  \intertext{and setting}
  \Nmap & := \Nmap^1 \fullstop
\end{align*}
In fact, because $\dis$ maps into the space of measures concentrated on the graph of a function, $\Nmap$ also maps into a smaller space, which we call $\FFF[1]$, and which is again defined by recursion down from $N-1$ to $1$:
\begin{align*}
  \FFF[N] & := \Pr {\X_N} \\
  \FFF[t] & := \FunP {\X_{t}} {\FFF[t+1]} \fullstop
\end{align*}
I.e.\ $\FFF$ is $\Pr {\XP 1}$ with all occurences of $\Pr {\cdot \times \cdot}$ replaced by $\FunP \cdot \cdot$.
Remember that we had
\begin{align*}
  \XP N & := \X_N \\
  \XP t & := \X_t \times \Pr {\XP {t+1}} \fullstop
\end{align*}
For convenience, let us also define
\begin{align*}
  \PPP[t] & := \Pr {\XP t} \fullstop
\end{align*}
\gobble{
For symmetry we also define
\begin{align*}
  \PPP[t] & := \Pr {\XP t}
  \intertext{and}
  \XF N & := \X_N \\
  \XF t & := \X_t \times \Pr {\XF {t+1}} \fullstop
\end{align*}
}

The fact that
\begin{align*}
  \Nmap^t : \Pr \Xb \to \FunP {\Xb^t} {\FFF[t+1]}
\end{align*}
and that therefore $\Nmap$ maps into $\FFF$ is a consequence of Lemma \ref{lem:functionsremainfunctions} below.

Finally, $\I'$ is defined as follows
\begin{align*}
  \I' & : \Pr {\Xb} \rightarrow \prod_{t=1}^{N-1} \FunP {\Xb^t} {\Pr{ \X_{t+1} }} \\
  \I'(\mu) & := (\I'_t(\mu))_t \\
  \I'_t & : \Pr {\Xb} \rightarrow \FunP {\Xb^t} {\Pr{ \X_{t+1} }} \\
  \I'_t & := \disint {\Xb^t} {\X_{t+1}} \circ \push{\proj_{\Xb^{t+1}}} \fullstop
\end{align*}
I.e. the reduced information topology, like the information topology, makes continuous predictions about the behaviour of the process after time $t$ given information about its behaviour up to time $t$, only now we are just predicting what the process will do in the next step, not for the rest of time.

$\I$, $\I'$ and $\Nmap$ are injective and therefore bijections onto their codomains. This means that the values of the maps $\K$, $\L$, $\H$ in diagram \eqref{diag:contmaps} as functions between sets are really already prescribed. The task consists in finding a representation for them which makes it clear that they are continuous.

\begin{lemma}
  \label{lem:functionsremainfunctions}
  $\disint \A {\B \times \Y}$ restricted to $\FunP {\A \times \B} \Y$ maps onto $\FunP<\!\big> \A { \FunP \B \Y }$.
\end{lemma}
\begin{proof}
  We first show that it maps into $\FunP<\big> \A { \FunP \B \Y }$.
  Let $\nu \in \FunP<\big> {\A \times \B} \Y$ and let $g : \A \times \B \rightarrow \Y$ be a function witnessing this fact, i.e.\ $\nu(f) = \int f(a,b,g(a,b)) \d \nu(a,b,\noarg)$.

  Let $\alpha := \disint \A {\B \times \Y} (\nu)$. Then
  $$\int \int 1_{g(a,b) \neq y} \d \beta(b, y) \d\alpha(a, \beta) = \int 1_{g(a,b) \neq y} \d \nu(a,b,y) = 0 \fullstop$$
  This means that for $\alpha$-a.a.\ $(a,\beta)$ we have $ \int 1_{g(a,b) \neq y} \d \beta(b, y) = 0 $, i.e.\ $\beta$ is concentrated on the graph of the function $b \mapsto g(a,b)$.

  To see that any $\alpha \in \FunP<\big> \A { \FunP \B \Y }$ can be obtained as the image of some $\nu \in \FunP {\A \times \B} \Y$ under $\disint \A {\B \times \Y}$, note that for such $\alpha$, by the existence of measurably dependent (classical) disintegrations (see for example \cite[Proposition 7.27]{BeSh78}), $\nu := \und \A {\B \times \Y} (\alpha) \in \FunP {\A \times \B} \Y$, and $\disint \A {\B \times \Y} (\nu) = \alpha$.
\end{proof}

\subsection{Homeomorphisms}
We give a plain language description of what follows in this section:

The continuity of $\L$ will be quite trivial, because we are just discarding information.

The components $\K_k : \FFF \rightarrow \FunP {\Xb^k} {\Pr {\Xb_{k+1}}}$ of the map $\K$ are obtained by \enquote*{folding} both the \enquote*{head} and the \enquote*{tail} of $\FFF$ using iterated application of the map $\undis$.
\begin{align*}
  \overbrace{ \mathscr F \Bigg( \X_1 \rightsquigarrow \mathscr F \bigg ( \dots \rightsquigarrow \mathscr F \Big ( \X_k }^{\text{head}} \rightsquigarrow \overbrace{\mathscr F \big ( \X_{k+1} \rightsquigarrow \mathscr F ( \dots \rightsquigarrow \mathscr P ( \X_N ) \dots ) \big ) \!\!\vphantom{\Bigg|}}^{\text{tail}} \,\, \Big ) \dots \bigg ) \Bigg )
\end{align*}
By continuity of $\undis$, it's easy to see that $\K_k$ is continuous.
To show that the map $\K$ with the components $\K_k$ is the map we are looking for, we basically show that 
\begin{align}
\label{eq:temp1321}
\I^{-1} \circ \K_k = \Nmap^{-1} \fullstop
\end{align}
$\Nmap^{-1}$ is again another way of \enquote*{folding} all of $\FFF$ using $\undis$ to arrive at $\Pr \Xb$. As $\I^{-1}$ is also $\undis$, showing \eqref{eq:temp1321} amounts to showing that these two different ways of \enquote*{folding} -- first the head and tail and then in a last step the junction between $k$ and $k+1$ on the one hand, and from front to back on the other hand -- do the same thing. This may be intuitively clear to the reader. The proof works by repeated application of Lemma \ref{lem:undass}, which represents one step of \enquote*{folding order doesn't matter}. Using Lemma \ref{lem:undass} the proof is completely analogous to the proof that for an operation $\star$ satisfying $ (a \star b) \star c = a \star ( b \star c )$, i.e.\ for an associative operation, one has
\begin{multline*}
  \pa<\big>{ \pa{ \dots \pa{ \pa{x_1 \star x_2} \star x_3} \star \dots } \star x_k} \star \pa<\big>{ \pa{ \dots \pa{ \pa{ x_{k+1} \star x_{k+2} } \star x_{k+3} } \star \dots } \star x_N } \\ = \pa<\big>{ \pa{ \dots \pa{ \pa{x_1 \star x_2} \star x_3} \star \dots } \star x_N} \fullstop
\end{multline*}
As we know, for such an operation any way of parenthesizing the multiplication of $N$ elements gives the same result. An analogous statement holds for $\undis$, though we do not formally state or prove this.

Finally, in Lemma \ref{lem:Hcont}, using Lemma \ref{lem:indkernel} as the main ingredient we prove the \enquote*{hard direction}, i.e.\ that $\H$ is continuous. If the continuity of $\L$ and $\K$ as informally described here seem obvious to the reader they may wish to skip ahead to Lemma \ref{lem:indkernel} and Lemma \ref{lem:Hcont}.

\begin{remark}
The reader interested in working out the details and analogies between \enquote*{folding} using $\undis$ and associative binary operations might be interested in reading about \emph{monads} in the context of Category Theory first. (See for example Chapter VI in \cite{MacLane}.) In fact, $(\Prob, \boldsymbol\eta, \boldsymbol\mu)$ forms a monad, where
\begin{align*}
  \boldsymbol\eta_\X : \X \rightarrow \Pr \X
\end{align*}
sends an element $x$ of $\X$ to the dirac measure at $x$ and 
\begin{align*}
  \boldsymbol\mu_\X & : \Pr { \Pr \X } \rightarrow \Pr \X \\
  \boldsymbol\mu_\X(\nu) & := f \mapsto \iint f(x) \d \nu'(x) \d\nu(\nu') \fullstop
\end{align*}
This monad is studied in a little more detail in \cite{Giry}.
$\undis$ can be obtained from $\boldsymbol\mu$ and a \emph{tensorial strength} $t_{\A,\B} : \A \times \Pr \B \rightarrow \Pr { \A \times \B }$ in the sense described for example in \cite{Moggi}.
\end{remark}

To show that $\L$ is continuous we will need the following lemma.
\begin{lemma}
  \label{lem:disnat}
  $\disint \A \B$ is natural in $\B$, i.e.\ for $f: \B \rightarrow \B'$ the following diagram commutes.
  \begin{align*}
    \begin{tikzpicture}[x=6cm,y=2cm]
      \node (PAB)   at (1,1) {$ \Pr { \A \times \B} $};
      \node (PAB')  at (0,1) {$ \Pr { \A \times \B'} $};
      \node (FAB)   at (1,0) {$ \FunP \A {\Pr \B} $};
      \node (FAB')  at (0,0) {$ \FunP \A {\Pr {\B'}} $};
      \draw[->] (PAB) edge node[right] {$ \disint \A \B $} (FAB);
      \draw[->] (PAB') edge node[left] {$ \disint \A {\B'} $} (FAB');
      \draw[->] (PAB) edge node[above] {$ \push{ \id \A \times f } $} (PAB');
      \draw[->] (FAB) edge node[below] {$ \push{ \id \A \times \push f } $} (FAB');
    \end{tikzpicture}
  \end{align*}
\end{lemma}
\begin{proof}
  This is just straigtforward calculation using the definitions.
\end{proof}

Applying Lemma \ref{lem:disnat} with $\A = \Xb^k$, $\B = \Xb_{k+1}$, $\B' = \X_{k+1}$ and $f = \proj_{\X_{k+1}} : \Xb_{k+1} \rightarrow \X_{k+1}$ we get that 
\begin{align*}
  \I'_k = \disint {\Xb^k} {\X_{k+1}} \circ \push{ \id {\Xb^k} \times \proj_{\X_{k+1}} } = \push{ \id {\Xb^k} \times \push{ \proj_{\X_{k+1}}} } \circ \disint {\Xb^k} {\Xb_{k+1}}
\end{align*}
Setting $\L_k := \push{ \id {\Xb^k} \times \push{ \proj_{\X_{k+1}}} }$ we get $ \I'_k = \L_k \circ \I_k$ and then setting $\L((\mu_k)_k) := (\L_k(\mu_k))_k$ gives $\I' = \L \circ \I$.

There is an analogue of Lemma \ref{lem:disnat} which we list here for completeness.
\begin{lemma}
  \label{lem:FFFandP}
  $\und \A \B : \Pr { \A \times \Pr \B } \rightarrow \Pr { \A \times \B }$ is natural in $\B$, i.e.\ for $f: \B \rightarrow \B'$ the following diagram commutes:
  \begin{align*}
    \begin{tikzpicture}[x=6cm,y=2cm]
    \node (PPB)  at (1,1) {$\Pr {\A \times \Pr \B}$};
    \node (PPBB) at (0,1) {$\Pr {\A \times \Pr {\B'}}$};
    \node (PB)   at (1,0) {$\Pr {\A \times \B}$};
    \node (PBB)  at (0,0) {$\Pr {\A \times \B'}$};
    \draw[->] (PPB) edge node[above] {$\Pr {\id \A \times \push f}$} (PPBB);
    \draw[->] (PB) edge node[below] {$\Pr {\id \A \times f}$} (PBB);
    \draw[->] (PPB) edge node[right] {$ \und \A \B $} (PB);
    \draw[->] (PPBB) edge node[left] {$ \und \A {\B'} $} (PBB);
    \end{tikzpicture}
  \end{align*}
  In particular, if $\B \subseteq \B'$ then $$\restr{\und \A {\B'}}{\Pr {\A \times \Pr \B}} = \und \A \B$$ if we regard $\Pr {\A \times \Pr \B}$ as a subset of $\Pr {\A \times \Pr {\B'}}$ by recursively using the recipe: \enquote*{if $\B$ is a subset of $\B'$, then we can view $\Pr \B$ as the subset of those $\mu \in \Pr {\B'}$ which are concentrated on $\B$}.
\end{lemma}
\begin{proof}
  Again this is just calculation.
\end{proof}
We already implicity used the \enquote*{in particular}-part of Lemma \ref{lem:FFFandP} when we said that $\Nmap$ can be regarded both as a map into $\Pr{\XP 1}$ and into $\FFF$ but the use there seemed too trivial to warrant much mention. There will be more such tacit uses.
\let\XF\XP

Now we show that $\K$ is continuous. We claim that it can be written as
\begin{align*}
  \K(\mu) = \pa{ \K_k(\mu) }_k
\end{align*}
where
\begin{multline*}
  \K_k = \push<\Big>{ \id{\Xb^k} \times \pa<\big>{ \und {\Xb_{k+1}^{N-1}} {\XF{N} }  \circ \dots \circ \und {\Xb_{k+1}^{k+2}} {\XF{k+3}} \circ \und {\X_{k+1}} {\XF{k+2}} } }   \circ \\    \und {\Xb^{k-1}} {\XF{k}} \circ \dots \circ \und {\Xb^2} {\XF{3}} \circ \und {\Xb^1} {\XF{2}} \text{ ,}
\end{multline*}
or without the dots, letting $\funprod$ denote concatenation of functions, e.g.\ $\funprod_{i=3}^1 f_i = f_3 \circ f_2 \circ f_1$:
\begin{align*}
  \K_k = \push<\Big>{ \id{\Xb^k} \times \pa<\big>{ {\textstyle \funprod_{i=N-1}^{k+1}} \und { \Xb_{k+1}^i } {\XF{i+1} } } }   \circ {\textstyle \funprod_{i=k-1}^1}  \und {\Xb^i} {\XF{i+1}} \fullstop
\end{align*}

To prove this we will repeatedly apply the following lemma.

\begin{lemma}[$\undis$ is \enquote*{associative}]
  \label{lem:undass}
  $\undis$ satisfies the following relation:
  $$ \und {\A \times \B} \C \circ \und \A {\B \times \Pr \C} = \und \A {\B \times \C} \circ \push { \id \A \times \und \B \C} $$
  These maps can be seen in the following commutative diagram.
  \begin{align*}
    \begin{tikzpicture}[x=6cm,y=2cm]
      \node (PAPBPC) at (1,1) {$\Pr{\A \times \Pr{\B \times \Pr \C}}$};
      \node (PABPC)  at (0,1) {$\Pr{\A \times \B \times \Pr \C}$};
      \node (PAPBC)  at (1,0) {$\Pr{\A \times \Pr{\B \times \C}}$};
      \node (PABC)   at (0,0) {$\Pr{\A \times \B \times \C}$};
      \draw[->] (PAPBPC) edge node[above] {$\und \A {\B \times \Pr \C}$} (PABPC);
      \draw[->] (PABPC)  edge node[left]  {$\und {\A \times \B} \C$} (PABC);
      \draw[->] (PAPBPC) edge node[right] {$\push{ \id \A \times \und \B \C}$} (PAPBC);
      \draw[->] (PAPBC)  edge node[below] {$\und \A {\B \times \C}$} (PABC);
    \end{tikzpicture}
  \end{align*}
\end{lemma}
\begin{proof}
  This is just expanding the definition. Both maps send a measure $\alpha \in \Pr { \A \times \Pr { \B \times \Pr \C } }$ to the measure $\mu$ with
  \begin{align*}
    \tsint f \d\mu = \tsint f(a,b,c) \d\gamma(c) \d\beta(b, \gamma) \d\alpha(a,\beta) \fullstop
  \end{align*}
\end{proof}

\newcommand{\Temp}{\mathcal T}
\begin{lemma}
  \label{lem:easydirection}
  The following relation holds.
  \begin{align}
    \label{eq:easydirection}
    \und {\Xb^k} {\Xb_{k+1}} \circ \K_k = {\textstyle \funprod_{i=N-1}^1} \und {\Xb^i} {\XF{i+1}}
  \end{align}
\end{lemma}
\begin{proof}
  Again, this is just repeated application of Lemma \ref{lem:undass}.
  Below we define $\Temp_l$ for $N \geq l \geq k$ and show that
  \begin{align}
    \label{eq:tempass}
    \und {\Xb^k} {\Xb_{k+1}} \circ \funprod_{i=N-1}^{k+1} \push{ \id{\Xb^k} \times \und {\Xb_{k+1}^i} {\XP{i+1} }} = \Temp_l
  \end{align}
  for all $N \geq l \geq k$ by showing $\Temp_l = \Temp_{l-1}$ for all $N \geq l > k$.
  The left hand side of \eqref{eq:tempass} is the left hand side of \eqref{eq:easydirection} with the common tail ${\textstyle \funprod_{i=k-1}^1}  \und {\Xb^i} {\XF{i+1}}$ of the left and right side in \eqref{eq:easydirection} dropped. $\Temp_k$ will be the right hand side of \eqref{eq:easydirection} with the common part dropped.
  \begin{align*}
    \Temp_l := \funprod_{i=N-1}^l \und {\Xb^i} {\XP{i+1}}  \quad  \circ \,
               \und {\Xb^k} {\Xb_{k+1}^l \times \PPP[l+1]}  \, \circ \quad
               \funprod_{i=l-1}^{k+1} \push{ \id{\Xb^k} \times \und {\Xb_{k+1}^i} {\XP{i+1}} }
  \end{align*}
  Here we regard $\funprod_r^s \dots$ with $r < s$ (an empty product in our context) as the identity function.
  For $l = N$ the first factor is an empty product and therefore clearly \eqref{eq:tempass} is true for $l = N$.
  To get from $\Temp_l$ to $\Temp_{l-1}$ we leave the first factor alone and apply Lemma \ref{lem:undass} with $\A = \Xb^k $, $\B = \Xb_{k+1}^{l-1}$ and $\C = \XP{l}$.
  This transforms 
  $$ \und {\Xb^k} {\Xb_{k+1}^l \times \PPP[l+1]} \circ \push{ \id{\Xb^k} \times \und {\Xb_{k+1}^{l-1}} {\XP{l}} } $$
  into
  $$ \und {\Xb^{l-1}} {\XP{l}} \circ   \und {\Xb^k} {\Xb_{k+1}^{l-1} \times \PPP[l]} $$
  and therefore $\Temp_l$ into $\Temp_{l-1}$.
\end{proof}

\begin{lemma}
  \label{lem:easyconclusion}
  The right hand triangle in \eqref{diag:contmaps} commutes, i.e.\ 
  \begin{align*}
    \K_k \circ \Nmap = \I_k \fullstop
  \end{align*}
\end{lemma}
\begin{proof}
  Prepending $\Nmap$ to \eqref{eq:easydirection} gives
  $$ \restr{\und {\Xb^k} {\Xb_{k+1}}}{\FunP {\Xb^k} {\Pr {\Xb_{k+1}}}} \circ \K_k \circ \Nmap = \id { \Pr \Xb } $$
  and appending $\I_k$ gives
  $$ \K_k \circ \Nmap = \I_k \fullstop $$
\end{proof}

Now we will show that $\H$ is continuous.
We will postpone the proof of Lemma~\ref{lem:indkernel} below, which is the crucial non-bookkeeping ingredient in the proof of Lemma~\ref{lem:Hcont} below, until the end of this section. The methods used in the proof of Lemma~\ref{lem:indkernel} differ significantly from the rest in this section and make use of the concept of the \emph{modulus of continuity} for measures, and results relating to it, introduced in the companion paper \cite{ModulusOfContinuity} to this one.

\begin{restatable}{lemma}{indkernel}
  \label{lem:indkernel}
  Let $$\dom {\J \A \B \Y} \subseteq \FunP<\big> \A {\Prob(\B)} \times \FunP<\big> {\A \times \B} \Y$$ be the set of all $(\mu', \mu)$ s.t.\  
  \begin{align}
    \label{eq:mu'mu}
    \und \A \B (\mu') = \marg {\mu} {\A \times \B} \fullstop
  \end{align}
  The function
  \begin{align*}
    \J \A \B \Y & : \dom {\J \A \B \Y} \rightarrow \FunP<\big> \A { \FunP \B \Y } \\
    \J \A \B \Y & (\mu',\mu) := \disint \A {\B \times \Y} (\mu)
  \end{align*}
  is continuous.
\end{restatable}

Clearly, as a function between sets, $\J \A \B \Y (\mu',\mu)$ only depends on $\mu$. But, as we know, $\disint \A {\B \times \Y}$ is \emph{not} continuous. Only when we refine the topology on the source space, which we encode by regarding $\J \A \B \Y$ as a map from the above subset of a product space, does it become continuous.

\begin{lemma}
  \label{lem:Hcont}
  $\H$ is continuous.
\end{lemma}
\begin{proof}
  We will inductively define
  \begin{align*}
    \H^k : \I'\br[]{\Pr\Xb} \rightarrow \Pr{ \Xb^k \times \PPP[k+1]}
  \end{align*}
  (again down from $N-1$ to $1$) so that they will be continuous by construction (and by virtue of Lemma \ref{lem:indkernel}). Also by construction, we will have $\H^k \circ \I' = \Nmap^k$.
  $\H$ will be $\H^1$ so that $\H \circ \I' = \Nmap$.

  Set $\H^{N-1} := \proj_{N-1}$, the projection from $\prod_{k=1}^{N-1} \FunP {\Xb^k} {\Pr{ \X_{k+1} }}$ onto the last factor. $\H^{N-1} \circ \I' = \I'_{N-1} = \disint {\Xb^{N-1}} {\X_N} = \Nmap^{N-1}$ by definition.
  Given $\H^{k+1}$ define
  \begin{align*}
    \H^k(\mu) := \J {\Xb^k} {\X_{k+1}} {\FFF[k+2]} \pa { \proj_k(\mu), \H^{k+1}(\mu) } \text{ ,}
  \end{align*}
  where $\proj_k$ is the projection from $\prod_{k=1}^{N-1} \FunP {\Xb^k} {\Pr{ \X_{k+1} }}$ onto the $k$-th factor.

  For this to be well-defined we need to check that for $\mu \in \I'\br[]{\Pr\Xb}$ we have $$\und {\Xb^k} {\X_{k+1}} (\proj_k(\mu)) = \push{\proj_{\Xb^{k+1}}}\pa{\H^{k+1}(\mu)} \fullstop$$
  I.e.\ for $\nu \in \Pr{\Xb}$ we want
  \begin{align*}
    \und {\Xb^k} {\X_{k+1}} (\proj_k(\I'(\nu))) = \push{\proj_{\Xb^{k+1}}}\pa{\H^{k+1}(\I'(\nu))}
  \end{align*}
  The composite of the maps on the left-hand side is equal to
  \begin{align*}
    \und {\Xb^k} {\X_{k+1}} \circ \I'_k = \und {\Xb^k} {\X_{k+1}} \circ \disint {\Xb^k} {\X_{k+1}} \circ \push{\proj_{\Xb^{k+1}}} = \push{\proj_{\Xb^{k+1}}} \fullstop
  \end{align*}
  On the right-hand side we get by induction hypothesis
  \begin{align}
    \label{eq:rhscalc}
    \push{\proj_{\Xb^{k+1}}} \circ \Nmap^{k+1} \fullstop
  \end{align}
  Using that $\push{\proj_\A} \circ \disint \A \B = \push{\proj_\A}$ we see for $l \geq k+1$
  \begin{multline*}
    \push{\proj_{\Xb^{k+1}}} \circ \push{\proj_{\Xb^{l}}} \circ \Nmap^l = \\ 
    \push{\proj_{\Xb^{k+1}}} \circ \push{\proj_{\Xb^{l}}} \circ \disint {\Xb^l} {\XF{l+1}} \circ \Nmap^{l+1} = \\
    \push{\proj_{\Xb^{k+1}}} \circ \push{\proj_{\Xb^{l}}} \circ \Nmap^{l+1} = \\
    \push{\proj_{\Xb^{k+1}}} \circ \push{\proj_{\Xb^{l+1}}} \circ \Nmap^{l+1} \text{ ,}
  \end{multline*}
  i.e.\ by induction \eqref{eq:rhscalc} is also equal to $ \push{\proj_{\Xb^{k+1}}} $.

  As a composite of continuous maps $\H^k$ is clearly continuous. (This is where we use Lemma \ref{lem:indkernel}.) As a map between sets $\H^k$ is just 
  \begin{align*}
    \disint {\Xb^k} {\XF{k+1}} \circ \H^{k+1} = \disint {\Xb^k} {\XF{k+1}} \circ \Nmap^{k+1} = \Nmap^k
  \end{align*}
  by induction hypothesis and definition of $\Nmap^k$.
\end{proof}

\subsection{Proof of \texorpdfstring{Lemma \ref{lem:indkernel}}{the induction step}}

In this part we prove Lemma \ref{lem:indkernel}. Here we use several of the ideas developed in the companion paper \cite{ModulusOfContinuity}. In particular we will need \cite[Lemma 4.2]{ModulusOfContinuity} which we reproduce below.

\begin{lemma}[{\cite[Lemma 4.2]{ModulusOfContinuity}}]
  \label{lem:helper}
  Let $\mu \in \FunP \X \Y$. For any $\epsilon > 0$ there is a $\delta > 0$ s.t. if
  \begin{align*}
    \nu \in \Pr { \X \times \Y } & \text{ with } \W \mu \nu < \delta \text{ and} \\
    \gamma \in \Cpl \mu \nu      & \text{ with } \tsint \D(x_1,x_2)^p \d\gamma(x_1,y_1,x_2,y_2) < \delta^p
  \end{align*}
  then
  \begin{align*}
    \tsint \D(y_1,y_2)^p \d\gamma(x_1,y_1,x_2,y_2) < \epsilon^p \fullstop
  \end{align*}
\end{lemma}

For easy reference we also restate Lemma \ref{lem:indkernel}.
\indkernel*

\begin{proof}[Proof of Lemma \ref{lem:indkernel}]
  Let $(\mu',\mu) \in \dom<> {\J \A \B \Y}$. Let $\epsilon > 0$.

  Choose $\delta > 0$ according to Lemma \ref{lem:helper} with $\X = \A \times \B$, i.e.\ s.t.\ for any $\nu \in \Pr { \A \times \B \times \Y }$ with $\W \mu \nu < \delta$ and any $\gamma \in \Cpl \mu \nu$ with $\tsint \D(a_1,a_2)^p + \D(b_1,b_2)^p \d\gamma(a_1,b_1,\noarg,a_2,b_2,\noarg) < \delta^p$ we have $\tsint \D(y_1,y_2)^p \d\gamma(\noarg,y_1,\noarg,y_2) < \epsilon^p$.

  Let $(\nu',\nu) \in \dom<>{\J \A \B \Y}$ with $\max(\D(\mu,\nu) , \D(\mu',\nu')) < \min(\delta , \epsilon)$.

  This means we can find $\gamma' \in \Cpl {\mu'} {\nu'}$ with
  \begin{align}
    \label{eq:selgamma'}
    \tsint \D(a_1,a_2)^p + \W<> {\hat b_1} {\hat b_2}^p \d\gamma'(a_1,\hat b_1, a_2, \hat b_2) < \min(\delta^p , \epsilon^p) \fullstop
  \end{align}

  \newcommand{\bmu}{\bar\mu}
  \newcommand{\bnu}{\bar\nu}
  Let $(a,b) \mapsto f_a(b)$ and $(a,b) \mapsto g_a(b): \A \times \B \rightarrow \Y$ be measurable functions on whose graph $\mu$ and $\nu$, respectively, are concentrated. Let $\bmu := \J \A \B \Y (\mu',\mu)$, $\bnu :=  \J \A \B \Y (\nu', \nu)$.
  
  \newcommand{\dmu}{\dot\mu}
  \newcommand{\dnu}{\dot\nu}
  As noted in the proof of Lemma \ref{lem:functionsremainfunctions} we know that for $\bmu$-a.a.\ $(a,\dmu)$ the measure $\dmu$ is concentrated on the graph of the function $f_a$ (and similarly for $\bnu$).
  This together with $\push{\id \A \times \push{\proj_\B}}(\bmu) = \mu'$ (which is a consequence of \eqref{eq:mu'mu}) implies that
  \begin{align*}
    \tsint h \d\bmu = \tsint h\pa{a,\push{\id \B \funcomma f_a}(\hat b)} \d\mu'(a,\hat b) 
  \end{align*}
  (again similarly for $\bnu$).

  \newcommand{\bgamma}{\bar\gamma}
  From this we see that the measure $\bgamma \in \Pr {\A \times \FunP \B \Y \times \A \times \FunP \B \Y}$ defined as
  \begin{align*}
    \tsint h \d\bgamma := 
    \tsint h\pa{a_1,\push{\id \B \funcomma f_{a_1}}(\hat b_1), a_2, \push{\id \B \funcomma g_{a_2}}(\hat b_2)} \d\gamma'(a_1,\hat b_1, a_2,\hat b_2)
  \end{align*}
  is in $\Cpl \bmu \bnu$.

  \newcommand{\hgamma}{\hat\gamma_{\hat b_1,\hat b_2}}
  We may measurably select almost-witnesses $\hgamma \in \Cpl<> {\hat b_1} {\hat b_2}$ for the distances $\W<> {\hat b_1} {\hat b_2}$ s.t.\ building on \eqref{eq:selgamma'} we get
  \begin{align}
    \label{eq:leftbound}
    \tsint \D(a_1,a_2)^p + \tsint \D(b_1,b_2)^p \d\hgamma(b_1,b_2) \d\gamma'(a_1,\hat b_1, a_2, \hat b_2) < \min(\delta^p , \epsilon^p) \fullstop
  \end{align}
  \newcommand\dolabel[1]{\stepcounter{equation}\tag{\theequation}\label{#1}}
  \newcommand{\hb}{\smash{\hat b}}
  Now
  \begin{align*}
    & \D(\bmu,\bnu)^p \leq \tsint \D_{\Pr { \A \times \Pr { \B \times \Y} }}^p \d \bgamma \\
    & = \tsint \D(a_1,a_2)^p + \W { \push{\id \B \funcomma f_{a_1}}(\hat b_1) } { \push{\id \B \funcomma g_{a_2}}(\hat b_2) } ^p \d\gamma'(a_1,\hb_1,a_1,\hb_2) \\
    & \leq \tsint \D(a_1,a_2)^p + \tsint \D(b_1,b_2)^p + \D\pa{f_{a_1}(b_1),g_{a_2}(b_2)}^p \d\hgamma(b_1,b_2) \d\gamma'(a_1,\hb_1,a_2,\hb_2) \\
    & = \tsint \D(a_1,a_2)^p + \D(b_1,b_2)^p + \D(y_1,y_2)^p \d\gamma(a_1,b_1,y_1,a_2,b_2,y_2) \dolabel{eq:tobebounded}
  \end{align*}
  where $\gamma \in \Cpl \mu \nu$ is defined as
  \begin{align*}
    \tsint h \d\gamma = \tsiint h\pa{a_1,b_1,f_{a_1}(b_1),a_2,b_2,g_{a_2}(b_2)} \d\hgamma(b_1,b_2) \d\gamma'(a_1,\hb_1,a_2,\hb_2) \fullstop
  \end{align*}
  The integral over the first two summands in \eqref{eq:tobebounded} is less than $\min(\delta^p , \epsilon^p)$ by \eqref{eq:leftbound}. By our choice of $\delta$ in the beginning this implies that the integral over the last summand is also less than $\epsilon^p$, so that overall
  \begin{align*}
    \D(\bmu,\bnu)^p < 2 \epsilon^p \fullstop
  \end{align*}
  Es $\epsilon$ was arbitrary this concludes the proof.
\end{proof}

\section{The symmetrized causal Wasserstein distance\texorpdfstring{ $\SCWp$}{}}
\label{sec:CausalAndAnti}

In this section we prove that the topology induced by $\SCWp$ is sandwiched between Hellwig's $\Wp$-information topology and the topology induced by $\AWp$, and therefore by what we have already seen in the previous section equal to both of them.
Our arguments in this section make explicit use of metrics. The reader who is only interested in the simpler version of our main theorem, Theorem \ref{thm:main_simple} may assume that $p = 1$ and that all metrics are bounded.

Remember that for $\mu, \nu \in \Pr \Xb$ we have
\begin{align}
  \label{eq:WcDef}
  \CWp(\mu,\nu)^p  & = \inf_{\substack{\gamma \in \Cpl \mu \nu \\ \gamma \text{ causal}}} \int \sum_{t=1}^N \D(x_t,y_t)^p \d\gamma((x_t)_t,(y_t)_t) \\
  \SCWp(\mu,\nu) & = \max\pa{\CWp(\mu,\nu),\CWp(\nu,\mu)} \\[2mm]
  \AWp(\mu,\nu)^p  & = \inf_{\substack{\gamma \in \Cpl \mu \nu \\ \gamma \text{ bicausal}}} \int \sum_{t=1}^N \D(x_t,y_t)^p \d\gamma((x_t)_t,(y_t)_t) \fullstop
\end{align}

In proving this we will take a slightly roundabout route. First we will focus on the case where $\Xb = \Xs_1 \times \Xs_2$ is the product of just two spaces, i.e.\ where we have only two time points. Moreover, for expositional purposes, let us for the moment assume that $\Xs_1$ and $\Xs_2$ are both compact. Generalizing from this setting will not be very hard.

In the compact, two-time-point case we will show equality of the two topologies in question by extending both to a larger (compact) space and showing equality of the topologies on that larger space.

In more detail:

When there are only two timepoints Hellwig's $\Wp$-information topology and the topology induced by $\AWp$ trivially coincide. Both are induced by emedding $\Pr {\Xs_1 \times \Xs_2}$ into $\Pr {\Xs_1 \times \Pr{\Xs_2}}$ via $\disint {\Xs_1} {\Xs_2}$. The latter space carries its standard metric $\D_{\Pr {\Xs_1 \times \Pr{\Xs_2}}}$, which -- as was already established in Theorem \ref{thm:Nembedding} in Section \ref{sec:DerErhabeneRaum} of the introduction -- is an extension of $\AWp$. To highlight this connection, in this section we will also refer to that metric as $\WWbc$. As a reminder, 
\begin{align*}
  \WWb \mu \nu ^p =  \inf_{\gamma \in \Cpl \mu \nu} \int \D(x_1,y_1)^p + \W {\xi_2} {\eta_2} ^p \d \gamma(x_1, \xi_2, y_1, \eta_2)
\end{align*}
where $\Wa$ is the normal Wasserstein distance (on $\Pr {\Xs_2}$ in this case).
We will find an extension $\WWcau$ of $\Wcau$ to $\Pr {\Xs_1 \times \Pr{\Xs_2}}$, which still satisfies all properties of a metric except for symmetry and which is dominated by $\WWbc$. Symmetrizing this extension gives a metric (which we will call $\WWcac$). The identity function from $\Pr {\Xs_1 \times \Pr{\Xs_2}}$ topologized with $\WWbc$ to $\Pr {\Xs_1 \times \Pr{\Xs_2}}$ topologized with $\WWcac$ will then be a continuous bijection from a compact space (this is where we use compactness of $\Xs_1$, $\Xs_2$) to a \emph{Hausdorff} space, i.e.\ a homeomorphism.

The next subsection will be devoted to finding an expression for the extension of $\Wcau$ to $\Pr {\Xs_1 \times \Pr{\Xs_2}}$ and proving that it satisfies all the properties mentioned above.

\begin{remark}
  When $\Xs_1$ contains no isolated points, because $\Pr {\Xs_1 \times \Pr{\Xs_2}}$ is the metric completion of $\Pr {\Xs_1 \times \Xs_2}$ w.r.t.\ $\Wbc$ and because the above properties imply that $\Wcau$ is (uniformly) continuous w.r.t.\ $\Wbc$, we have already uniquely identified $\WWcau$. Still, we want to find an expression that allows us to work with $\WWcau$ and in particular that allows us to prove that $\WWcac$ is a metric and not just a pseudometric, i.e.\ that the induced topology is in fact Hausdorff.
  This is exactly what we gain from assuming compact base spaces and passing to the completion: instead of having to find a lower bound for $\Wca \mu \nu$ in terms of $\Wb \mu \nu$ (and possibly $\mu$) we now just have to prove that if $\mu \neq \nu$ then $\WWca \mu \nu > 0$.
\end{remark}
  
  {\color{black}For definiteness we note that  we do \emph{not} assume,  compactness of any space in the following.}

\subsection{Extending the causal \texorpdfstring{`distance'}{'distance'}}

So now we are working with two Polish metric spaces $\Xs_1$, $\Xs_2$.
Remember that we denote the \enquote*{canonical process} on $\Xb := \Xs_1 \times \Xs_2$ by $(\Xr_i)_{i=1,2}$, i.e.\ $\Xr_i : \Xb \rightarrow \Xs_i$ is the projection onto the $i$-th coordinate.

To differentiate between the different roles that $\Xb$ may play - i.e.\ is it the space for the left measure $\mu$ or the right measure $\nu$ when measuring the \enquote*{distance} $\Wc \mu \nu$ - we will also refer to $\Xb$, $\Xs_i$ by the aliases $\Yb$, $\Ys_i$ respectively. (And later $\Zb$, $\Zs_i$ as well.)
Analogously, we have $\Yr_i : \Yb \rightarrow \Ys_i$. (And $\Zr_i : \Zb \rightarrow \Zs_i$.)

In this section we will repeatedly make use of the following construction:
\begin{definition}
  \label{def:dpro}
  Let $\A$, $\B$, $\C$ be Polish metric spaces.
  Let $\mu \in \Pr {\A \times \B}$ and $\nu \in \Pr {\B \times \C}$ with $\marg \mu \B = \marg \nu \B$.
  We define
  \begin{align*}
    \mu \dpro{\B} \nu \in \Pr {\A \times \B \times \C}
  \end{align*}
  as the measure given by
  \begin{align}
    \label{eq:dpro}
    \begin{aligned}
    \tsint h \d(\mu \dpro \B \nu) 
    & := \tsint h(a,b,c) \d\nu_b(c) \d\mu(a,b) \\
    & \phantom{:}= \tsint h(a,b,c) \d\mu_b(a) \d\nu(b,c)
    \end{aligned}
  \end{align}
  where $b \mapsto \nu_b$ is a disintegration of $\nu$ w.r.t.\ $\B$ and similarly for $\mu$.

  We further define
  \begin{align*}
    \mu \mcmo{\B} \nu := \marg {\pa{\mu \dpro{\B} \nu}} {\A \times \C} \in \Pr {\A \times \C} \fullstop
  \end{align*}
\end{definition}

\begin{remark}
  If $\mu$ is a probability on $\S A \times \S B$ and $\nu$ is a probability on $\S B \times \S C$, another way of saying what $\mu \dpro{\S B} \nu$ is, is to state that it is a probability on $\S A \times \S B \times \S C$ s.t.\ the law of $(\RV A, \RV B)$ is equal to $\mu$, the law of $(\RV B, \RV C)$ is equal to $\nu$ (where per our convention $\RV A$ is the projection onto $\S A$, etc.), and $\RV A$ is conditionally independent from $\RV C$ given $\RV B$. (For the notion of conditional independence see for example \cite[Definition II.43]{DeMeA}.)
  
  Another helpful intuition comes from looking at the case where $\mu \in \FunP \A \B$ is concentrated on the graph of some measurable function $f: \A \rightarrow \B$ and $\nu \in \FunP \B \C$ is concentrated on the graph of a measurable function $g: \B \rightarrow \C$.
  $\mu \mcmo \B \nu$ is then concentrated on the graph of $g \circ f : \A \rightarrow \C$.
  In some contexts $g \circ f$ is also written as $f \mcmp g$, which is where we borrowed the symbol from.
\end{remark}

\begin{remark}
  We will often encounter the situation that one of the factors $\A$, $\B$ or $\C$ in Definition \ref{def:dpro} is itself a product of spaces and the individual factors may not always be so nicely sorted. We will rely on naming in the subscript the space(s) along which to join the measures $\mu$ and $\nu$. For example if $\mu \in \Pr {\A_1 \times \B_1 \times \A_1 \times \B_2}$ and $\nu \in \Pr {\B_2 \times \C_1 \times \B_1 \times \C_2}$ we might write
  \begin{align*}
    \mu \dpro{\B_1,\B_2} \nu \in \Pr {\A_1 \times \B_1 \times \A_2 \times \B_2 \times \C_1 \times \C_2}
  \end{align*}
  to refer to the measure that we get when in \eqref{eq:dpro} we use $(b_1,b_2) \in \B_1 \times \B_2$ as the middle variable $b$.
  We will not be systematic about the order of the factors in the resulting product space on which e.g.\ $\mu \dpro{\B_1,\B_2} \nu$ is a measure, again relying on naming our spaces for disambiguation.
\end{remark}

For future reference we paraphrase the definition of a causal transport plan given in \eqref{def:adapted.coupling} in the introduction.
\begin{lemma}
  \label{lem:causaldef}
  Let $\mu$ be a measure on $\Xb = \Xs_1 \times \Xs_2$ and $\nu$ be a measure on $\Yb = \Ys_1 \times \Ys_2$. $\gamma \in \Cpl \mu \nu$ is a causal transference plan from $\mu$ to $\nu$ iff under $\gamma$
  \begin{align*}
    \text{$\Xr_2$ and $\Yr_1$ are conditionally independent given $\Xr_1$.}
  \end{align*}
\end{lemma}
\begin{proof}
  One way of formulating conditional independence is as in \eqref{def:adapted.coupling}, see for example \cite[Definition II.43, Theorem II.45]{DeMeA}.
\end{proof}

In other words, $\gamma \in \Cpl \mu \nu$ is a causal transference plan iff $\marg \gamma {\X_1,\X_2,\Y_1} = \mu \dpro{\X_1} \marg \gamma {\X_1,\Y_1}$.

We start by reexpressing $\Wcau$ in different ways until we find one which also makes sense in $\Pr { \Xs_1 \times \Pr {\Xs_2}}$.

Let $\mu \in \Pr \Xb$ and $\nu \in \Pr \Yb$. Then

\newcommand{\Cstr}[1]{C_{#1}}
\begin{align*}
  \Wc \mu \nu ^p & = \inf_{\substack{\gamma \in \Cpl \mu \nu \\ \gamma \text{ causal}}} \int \D(x_1,y_1)^p + \D(x_2,y_2)^p \d \gamma(x_1,x_2,y_1,y_2) \\
      & = \inf_{\gamma \in \Cstr 1} \int \D(x_1,y_1)^p + \D(x_2,y_2)^p \d \gamma(x_1,x_2,y_1,y_2)
    \intertext{ where }
    \Cstr 1 & = \set {\gamma \in \Cpl \mu \nu}[\gamma = \pa{ \mu \dpro{\Xs_1} \marg \gamma {\Xs_1,\Ys_1} } \dpro{\Xs_2,\Ys_1} \marg \gamma {\Xs_2, \Ys_1, \Ys_2}] \fullstop
\end{align*}
This is true because, on the one hand clearly a $\gamma \in \Cstr 1$ is causal by Lemma \ref{lem:causaldef} and the alternative characterization of $\dpro{\Xs_1}$. On the other hand, given any causal $\gamma \in \Cpl \mu \nu$, again by Lemma \ref{lem:causaldef}, $\marg \gamma {\Xs_1,\Xs_2,\Ys_1} = \mu \dpro{\Xs_1} \marg \gamma {\Xs_1,\Ys_1}$, and we may define $\gamma' := \pa{ \mu \dpro{\Xs_1} \marg \gamma {\Xs_1,\Ys_1} } \dpro{\Xs_2,\Ys_1} \marg \gamma {\Xs_2, \Ys_1, \Ys_2} \enskip \in \Cpl \mu \nu$.
Now $\marg \gamma {\Xs_1,\Xs_2,\Ys_1} = \marg {\gamma'} {\Xs_1,\Xs_2,\Ys_1}$ and $\marg \gamma {\Xs_2,\Ys_1,\Ys_2} = \marg {\gamma'} {\Xs_2,\Ys_1,\Ys_2}$, so in particular
\begin{multline*}
  \int \D(x_1,y_1)^p + \D(x_2,y_2)^p \d \gamma(x_1,x_2,y_1,y_2) = \\ \int \D(x_1,y_1)^p + \D(x_2,y_2)^p \d \gamma'(x_1,x_2,y_1,y_2) \fullstop
\end{multline*}

We may name the different building blocks of $\gamma \in \Cstr 1$ to get
\begin{align*}
  \Wc \mu \nu ^p = 
  \inf_{(\gamma,\beta) \in \Cstr 2} \int \D(x_1,y_1)^p \d \gamma(x_1,y_1) + \int \D(x_2,y_2)^p \d \beta(y_1,x_2,y_2)
\end{align*}
with
\begin{multline*}
  \Cstr 2 = \set<\Big>{(\gamma,\beta) \in \Cpl{\marg \mu {\Xs_1}}{\marg \nu {\Ys_1}} \times \Pr {\Ys_1 \times \Xs_2 \times \Ys_2}}[\\\marg \beta {\Xs_2, \Ys_1} = \mu \mcmo{\Xs_1} \gamma \text{ and } \marg \beta {\Ys_1,\Ys_2} = \nu] \text{ ,}
\end{multline*}
i.e.\ there is a bijection between $\Cstr 1$ and $\Cstr 2$ given by sending $\gamma' \in \Cstr 1$ to $(\gamma, \beta) \in \Cstr 2$ where $\gamma := \marg {\gamma'} {\Xs_1,\Ys_1}$, $\beta := \marg {\gamma'} {\Xs_2,\Ys_1,\Ys_2}$, and, in the other direction, by sending $(\gamma, \beta) \in \Cstr 2$ to $\gamma' := \pa{ \mu \dpro{\Xs_1} \gamma } \dpro{\Xs_2,\Ys_1} \beta$.

We can apply the bijection $\diso {\Ys_1} : \Pr {\Ys_1 \times \Xs_2 \times \Ys_2} \rightarrow \FunP {\Ys_1} {\Pr {\Xs_2 \times \Ys_2}}$ to $\beta$. Translating the conditions on $(\gamma,\beta) \in \Cstr 2$ to conditions on $(\gamma, \diso {\Ys_1} (\beta))$ we arrive at
\begin{align*}
  \Wc \mu \nu ^p = 
  \inf_{(\gamma,\beta) \in \Cstr 3} \int \D(x_1, y_1)^p \d \gamma + \int \int \D(x_2,y_2)^p \d \beta'(x_2,y_2) \d \beta(y_1, \beta') 
\end{align*}
where
\begin{multline*}
  \Cstr 3 = \set<\Big>{(\gamma,\beta) \in \Cpl{\marg \mu {\Xs_1}}{\marg \nu {\Ys_1}} \times \FunP {\Ys_1} {\Pr {\Xs_2 \times \Ys_2}}}[ \\ \push{\id{\Ys_1} \times \push{\Yr_2}}(\beta) = \diso {\Ys_1} (\nu) \text{ and } \\ \push{\id{\Ys_1} \times \push{\Xr_2}}(\beta) = \diso {\Ys_1} \pa{ \gamma \mcmo{\Xs_1} \mu } ] \fullstop
\end{multline*}

Let $(\gamma, \beta) \in \Cstr 3$ and let $(y_1, \beta') \mapsto \tilde \beta ' _{y_1, \beta'}$ be a measurable mapping with $\tilde \beta ' _{y_1, \beta'} \in \Cpl {\marg {\beta'} {\Xs_2}} {\marg {\beta'} {\Ys_2}}$ for $\beta$-a.a.\ $(y_1,\beta')$. Then we have that also $(\gamma,\tilde\beta) \in \Cstr 3$, where $\tilde\beta \in \FunP {\Ys_1} {\Pr {\Xs_2 \times \Ys_2}}$ is defined by
\begin{align*}
  \tilde\beta := f \mapsto \int f(y_1,\tilde\beta'_{y_1,\beta'}) \d \beta(y_1,\beta') \fullstop
\end{align*}

By employing a $\beta$-a.e.\ measurable selector this implies that
\begin{align*}
  \Wc \mu \nu ^p & = 
  \inf_{(\gamma,\beta) \in \Cstr 3} {\textstyle \int} \D(x_1, y_1)^p \d \gamma + {\textstyle \int} \quad \inf_{\raisebox{0pt}[1.1em][0em]{}\mathclap{\substack{\tilde \beta' \in \\ \Cpl {\marg {\beta'} {\Xs_2}} {\marg {\beta'} {\Ys_2}}}}} \quad {\textstyle \int} \D(x_2,y_2)^p \d \tilde\beta'(x_2,y_2) \d \beta(y_1, \beta') \\[2mm]
  & = \inf_{(\gamma,\beta) \in \Cstr 3} {\textstyle \int} \D(x_1, y_1)^p \d \gamma + {\textstyle \int} \quad \W {\marg {\beta'} {\Xs_2}} {\marg {\beta'} {\Ys_2}} ^p \d \beta(y_1, \beta')
  \fullstop
\end{align*}

We need
\begin{lemma}
  \label{lem:conddetimplcondind}
  If $\kappa \in \Pr {\S A \times \S B}$ and $\lambda \in \FunP {\S B} {\S C}$ then the only measure $\eta \in \Pr {\S A \times \S B \times \S C}$ with $\marg \eta {\S A \times \S B} = \kappa$ and $\marg \eta {\S B \times \S C} = \lambda$ is $\kappa \dpro{\S B} \lambda$.
\end{lemma}
\begin{proof}
  If $\eta$ satisfies the properties above and $b \mapsto \kappa_b$, $b \mapsto \lambda_b$ are (classical) disintegrations of $\kappa$, $\lambda$ w.r.t.\ $\B$, then a (classical) disintegration $b \mapsto \eta_b$ of $\eta$ w.r.t.\ $\B$ has to satisfy $\marg {{\eta_b}} \A = \kappa_b$ and $\marg {{\eta_b}} \C = \lambda_b$ a.s. As $\lambda_b$ is a Dirac measure a.s.\ this forces $\eta_b$ to be $\kappa_b \mprod \lambda_b$ almost surely.
\end{proof}

This implies that for $(\gamma, \beta) \in \Cstr 3$ the distribution of
\begin{align}
  \label{eq:cstr4}
  (y_1,\beta') \mapsto (y_1, {\marg {\beta'} {\Xs_2}}, {\marg {\beta'} {\Ys_2}})  
\end{align}
under $\beta$ is already determined by $\gamma$, i.e.\ because the distribution of $(y_1,\beta') \mapsto (y_1, {\marg {\beta'} {\Xs_2}})$ is $\diso {\Ys_1} \!\!\pa{ \gamma \mcmo{\Xs_1} \mu }$ and the distribution of $(y_1,\beta') \mapsto (y_1, {\marg {\beta'} {\Ys_2}})$ is $\diso {\Ys_1} \!(\nu)$,
the distribution of \eqref{eq:cstr4} under $\beta$ must be equal to
\begin{align*}
  \diso {\Ys_1} \!\!\pa{ \gamma \mcmo{\Xs_1} \mu } \dpro{\Ys_1} \diso {\Ys_1} \!(\nu) \fullstop
\end{align*}

This means that we may get rid of $\beta$:

\begin{multline*}
  \Wc \mu \nu ^p =
  \inf_{\mathclap{\gamma \in \Cpl { \marg \mu {\Xs_1} } { \marg \nu {\Ys_1} } }} \quad {\textstyle \int} \D(x_1,y_1)^p \d \gamma \\ + {\textstyle \int} \W {\mu'} {\nu'} ^p \d \pa { \diso {\Ys_1} \!\!\pa{ \gamma \mcmo{\Xs_1} \mu } \dpro{\Ys_1} \diso {\Ys_1} \!(\nu) } (y_1,\mu',\nu')
\end{multline*}

For the final step we need another lemma:

\begin{lemma}
  Let $\lambda \in \Pr { \As \times \Bs }$ and $\beta \in \Pr { \Bs \times \Cs }$. Let $\CCr$ denote the projection onto $\Pr \Cs$. Then 
  \begin{align*}
    \diso \As \pa { \lambda \mcmo \Bs \beta }
  \end{align*}
  is equal to the distribution of
  \begin{align*}
    (\Ar, \E[\eta]{\CCr}[\Ar] )
    & \text{ under }
    \eta := \lambda \dpro \Bs \diso \Bs (\beta) \fullstop
  \end{align*}
\end{lemma}
\begin{proof}
  Let $a \mapsto \lambda_a$ be a version of the (classical) disintegration of $\lambda$ w.r.t.\ $\A$ and let $b \mapsto \beta_b$ be a disintegration of $\beta$ w.r.t.\ $\B$.

  As one easily checks, a version of the (classical) disintegration of $\lambda \mcmo \Bs \beta$ w.r.t.\ $\A$ is given by $a \mapsto \int \beta_b \d \lambda_a(b)$, so that $\diso \As \pa { \lambda \mcmo \Bs \beta }$ is equal to $$\push<\Big> {a \mapsto (a, \tsint \beta_b \d \lambda_a(b))} (\marg \lambda \As) \fullstop$$

  By the same argument a version of the disintegration of $\lambda \mcmo \Bs \diso \Bs (\beta)$ w.r.t.\ $\A$ is given by $h := a \mapsto \int \diso \Bs (\beta)_b \d \lambda_a(b)$, where $b \mapsto \diso \Bs (\beta)_b$ is a disintegration of $\diso \Bs (\beta)$ w.r.t.\ $\B$. But such a disintegration is given by $b \mapsto \delta_{\beta_b}$, (where $\delta_{\beta_b}$ is the dirac measure at $\beta_b$). So $h = a \mapsto \int \delta_{\beta_b} \d \lambda_a(b)$. This means (a version of) $\E[\eta]{\CCr}[\Ar]$ is given by
  \begin{align*}
    \E[\eta]{\CCr}[\Ar](a,\noarg,\noarg)
    = \tsint \hat c \d (\tsint \delta_{\beta_b} \d \lambda_a(b))(\hat c)
    = \tsiint \hat c \d \delta_{\beta_b}(\hat c) \d \lambda_a(b) = \tsint \beta_b \d \lambda_a(b) \text{ ,}
  \end{align*}
  so that the distribution of $(\Ar, \E[\eta]{\CCr}[\Ar] )$ under $\eta$ is also given by $$\push<\Big> {a \mapsto (a, \tsint \beta_b \d \lambda_a(b))} (\marg \lambda \As) \fullstop$$
\end{proof}

Using this lemma with $\As = \Ys_1$, $\Bs = \Xs_1$, $\Cs = \Xs_2$, $\lambda = \gamma$, $\beta = \mu$ and writing $\XXr$, $\YYr$ for the projections onto $\Pr {\Xs_2}$, $\Pr {\Ys_2}$ respectively, we find:
\begin{align*}
  \Wc \mu \nu ^p = \quad
  \inf_{\mathclap{\gamma \in \Cpl { \marg \mu {\Xs_1} } { \marg \nu {\Ys_1} } }} \quad \E[\gamma]{\D(\Xr_1, \Yr_1)^p} + \E[\eta(\gamma)]{\W {\E[\eta(\gamma)]{\XXr}[\Yr_1]} {\YYr} ^p}
\end{align*}
where $\eta(\gamma) := \diso {\Xs_1} (\mu) \dpro {\Xs_1} \gamma \dpro {\Ys_1} \diso {\Ys_1} (\nu)$.

By Lemma \ref{lem:conddetimplcondind} the function $\eta : \Cpl { \marg \mu {\Xs_1} } { \marg \nu {\Ys_1} } \rightarrow \Cpl {\diso {\Xs_1} (\mu)} {\diso {\Ys_1} (\nu)}$ is a bijection, so we may as well write
\begin{align*}
  \Wc \mu \nu ^p = \quad
  \inf_{\mathclap{\gamma \in \Cpl {\diso {\Xs_1} (\mu)} {\diso {\Ys_1} (\nu)}}} \quad \E[\gamma]{\D(\Xr_1, \Yr_1)^p} + \E[\gamma]{\W {\E[\gamma]{\XXr}[\Yr_1]} {\YYr} ^p} \fullstop
\end{align*}

Finally, under any $\gamma \in \Cpl {\diso {\Xs_1} (\mu)} {\diso {\Ys_1} (\nu)}$ we know that $\YYr$ is almost surely equal to a function of $\Yr_1$, so that the completions of the sigma-algebras generated by $\Yr_1$ and $\VYr := (\Yr_1,\YYr)$ respectively are equal. This means that $\E[\gamma]{\XXr}[\Yr_1] = \E[\gamma]{\XXr}[\VYr]$ a.s.\ and we arrive at our final expression for $\Wc \mu \nu$:
\begin{align*}
  \Wc \mu \nu = \inf_{\gamma \in \Cpl {\diso {\Xs_1} (\mu)} {\diso {\Ys_1} (\nu)}} \pa<\Big>{ \E<\big>[\gamma]{\D(\Xr_1, \Yr_1)^p + \Wa \pa{\E[\gamma]{\XXr}[\VYr], \YYr} ^p} }^{1/p} \fullstop
\end{align*}

Now this expression is trivial to generalize to $\mu \in \Pr {\Xs_1 \times \Pr {\Xs_2}}$ and $\nu \in \Pr {\Ys_1 \times \Pr {\Ys_2}}$, i.e.\ for such $\mu$, $\nu$ we set
\begin{align}
  \label{eq:WcLifted}
    \WWc \mu \nu & := \inf_{\gamma \in \Cpl \mu \nu} \pa<\bigg>{ \E<\Big>[\gamma]{\D(\Xr_1, \Yr_1)^p + \Wa\pa<\big>{ \E[\gamma]{\XXr}[\VYr], \YYr} ^p} }^{1/p} \fullstop
\end{align}

To summarize our discussion up to this point:
\begin{lemma} The function
  \begin{align*}
    \WWcau : \Pr {\Xs_1 \times \Pr {\Xs_2}}^2 \rightarrow \R_+
  \end{align*}
  as defined in \eqref{eq:WcLifted} is really an extension of
  \begin{align*}
    \Wcau : \Pr {\Xs_1 \times \Xs_2}^2 \rightarrow \R_+
  \end{align*}
  as defined in \eqref{eq:WcDef} (when $\Pr {\Xs_1 \times \Xs_2}$ is embedded into $\Pr {\Xs_1 \times \Pr {\Xs_2}}$ via $\diso{\Xs_1}$).
\end{lemma}

Next we promised to show
\begin{lemma}
  \label{lem:WcleqWb}
  $\WWcau$ is bounded by $\WWbc$, i.e.\ 
\begin{align*}
  \WWc \mu \nu \leq \inf_{\gamma \in \Cpl \mu \nu} \pa<\Big>{ \E<\big>[\gamma]{\D(\Xr_1, \Yr_1)^p + \Wa (\XXr,\YYr)^p} }^{1/p} = \WWb \mu \nu \fullstop
\end{align*}
\end{lemma}
\begin{proof}
  By the conditional version of Jensen's inequality applied to the convex function $(\hat x, \hat y) \mapsto \W {\hat x} {\hat y} ^p$ we have
  \begin{align*}
    \Wa \pa{ \E[\gamma]{\XXr}[\VYr] , \YYr }^p = \Wa \pa{ \E[\gamma]{ (\XXr, \YYr) }[\VYr] }^p \leq \E[\gamma]{ \W \XXr \YYr ^p}[\VYr] \fullstop
  \end{align*}
\end{proof}

\begin{remark}
For the reader who may be sceptical of whether Jensen's inequality holds in this rather unusual setting, where we have a convex function
\begin{align*}
  \Wa : \Pr {\Xs_2} \times \Pr {\Ys_2} \rightarrow \R_+
\end{align*}
and conditional expectations on spaces of measures we remark that for the Wasserstein distance in particular this is very easy to check. The proof is just integrating transport plans between $\XXr$ and $\YYr$ w.r.t.\ the distribution of these conditioned on $\VYr$ (in this case) to get transport plans between $\E[\gamma]{\XXr}[\VYr]$ and $\E[\gamma]{\YYr}[\VYr]$.
\end{remark}

\begin{lemma}
  \label{lem:WcTriangle}
  Let $\mu, \nu, \lambda \in \Pr {\Xs_1 \times \Pr {\Xs_2}}$. Then
  \begin{align*}
    \WWc \mu \lambda \leq \WWc \mu \nu + \WWc \nu \lambda \fullstop
  \end{align*}
\end{lemma}
\begin{proof}
  Using our naming convention we have $$\mu \in \Pr {\Xs_1 \times \Pr {\Xs_2}},\quad \nu \in \Pr {\Ys_1 \times \Pr {\Ys_2}},\quad \lambda \in \Pr {\Zs_1 \times \Pr {\Zs_2}} \fullstop$$
  We denote the projections onto $\Pr {\Xs_2}$, $\Pr {\Ys_2}$, $\Pr {\Zs_2}$ by $\XXr$, $\YYr$, $\ZZr$ respectively. $\VYr = (\Yr_1,\YYr)$, $\VZr := (\Zr_1, \ZZr)$.

  Let $\gamma \in \Cpl \mu \nu$ and $\eta \in \Cpl \nu \lambda$. In the following let $\ExpS$ refer to (conditional) expectation w.r.t.\ $\kappa := \gamma \dpro {\Ys_1, \Pr {\Ys_2}} \eta$, and let $\Lpnorm{\cdot}$ refer to the $\Lp$-norm w.r.t.\ $\kappa$.

  Combining the triangle inequalities for $\D$, $\Wa$ and the $\Lpnorm \cdot$ we get
  \begin{align}
    \label{eq:LpD}
    \Lpnorm{\D(\Xr_1,\Zr_1)} & \leq \Lpnorm{\D(\Xr_1,\Yr_1)} + \Lpnorm{\D(\Yr_1,\Zr_1)} \\
    \label{eq:LpW'}
    \Lpnorm{\W {\E{\XXr}[\VZr]} {\ZZr}} & \leq \Lpnorm{\Wa \pa{\E{(\XXr,\YYr)}[\VZr]}} + \Lpnorm{\W {\E{\YYr}[\VZr]} {\ZZr}}
  \end{align}

  By the conditional Jensen inequality
  \begin{align*}
    \Wa \pa{\E{(\XXr,\YYr)}[\VZr]}^p & =
    \Wa \pa<\bigg>{ \E<\Big>{\E<\big>{(\XXr,\YYr)}[\VYr,\VZr]}[\VZr] }^p \\ & \leq \E<\bigg>{ \Wa \pa<\Big> {\E<\big>{(\XXr,\YYr)}[\VYr,\VZr]}^p }[\VZr]
  \end{align*}
  and therefore
  \begin{align*}
    \Lpnorm{ \Wa \pa{ \E{(\XXr,\YYr)}[\VZr]} }^p \leq \Lpnorm{ \Wa \pa{ \E{(\XXr,\YYr)}[\VYr,\VZr] }^p } \fullstop
  \end{align*}
  By construction, $(\XXr,\YYr)$ is conditionally independent from $\VZr$ given $\VYr$, so that $\E{(\XXr,\YYr)}[\VYr,\VZr] = \E{(\XXr,\YYr)}[\VYr]$ (this basic fact about conditional independence can be found for example as Theorem 45 in \cite{DeMeA}).
  Combining this with \eqref{eq:LpW'} gives
  \begin{align}
    \label{eq:LpW}
    \Lpnorm{\W {\E{\XXr}[\VZr]} {\ZZr}} & \leq \Lpnorm{\W {\E{\XXr}[\VYr]} {\YYr} } + \Lpnorm{\W {\E{\YYr}[\VZr]} {\ZZr}} \fullstop
  \end{align}
  Putting together \eqref{eq:LpD} and \eqref{eq:LpW} with the triangle inequality for $\ellp$ we get
  \begin{align*}
    \Wc \mu \lambda & = \pa{ \Lpnorm{ \D(\Xr_1,\Zr_1) }^p + \Lpnorm{ \W {\E{\XXr}[\VZr]} {\ZZr} }^p }^{1/p} \\
     & \leq \pa{ \Lpnorm{ \D(\Xr_1,\Yr_1) }^p + \Lpnorm{ \W {\E{\XXr}[\VYr]} {\YYr} }^p }^{1/p} \\
     & \qquad + \pa{ \Lpnorm{ \D(\Yr_1,\Zr_1) }^p + \Lpnorm{ \W {\E{\YYr}[\VZr]} {\ZZr} }^p }^{1/p} \\
     & = \Wc \mu \nu + \Wc \nu \lambda \fullstop
  \end{align*}
\end{proof}

\begin{lemma}
  $\WWcau$ is uniformly continuous w.r.t.\ $\WWbc$ on $\Pr {\Xs_1 \times \Pr {\Xs_2}}^2$.
\end{lemma}
\begin{proof}
  Let $\mu, \nu, \mu', \nu' \in \Pr {\Xs_1 \times \Pr {\Xs_2}}$. We repeatedly use Lemma \ref{lem:WcTriangle}:
  \begin{align*}
    \WWc \mu \nu \leq \WWc {\mu} {\nu'} + \WWc {\nu'} {\nu} \leq \WWc {\mu} {\mu'} + \WWc {\mu'} {\nu'} + \WWc {\nu'} {\nu}
  \end{align*}
  therefore
  \begin{align*}
    \WWc \mu \nu - \WWc {\mu'} {\nu'} \leq \WWc {\mu} {\mu'} + \WWc {\nu'} {\nu} \fullstop
  \end{align*}
  Switching the roles of $(\mu,\nu)$ and $(\mu',\nu')$ implies
  \begin{align*}
    \lvert \WWc \mu \nu & - \WWc {\mu'} {\nu'} \rvert \\
     & \leq \max\pa{ \WWc {\mu} {\mu'} , \WWc {\mu'} {\mu} } + \max\pa{ \WWc {\nu} {\nu'} , \WWc {\nu'} {\nu} } \\
     & \leq \WWb {\mu} {\mu'} + \WWb {\nu} {\nu'} \fullstop
  \end{align*}
\end{proof}

\begin{lemma}
  The infimum in \eqref{eq:WcLifted} is attained.
\end{lemma}
\begin{proof}
  This is an application of \cite[Theorem 1.2]{BaBePa18}.

  For self-containedness and because it's a nice application of the nested distance, we also sketch the argument here.
  We know that $\Cpl \mu \nu$ is compact. The problem is that $\gamma \mapsto \E[\gamma]{\W {\E[\gamma]{\XXr}[\VYr]} {\VYr} ^p}$ is not (lower semi-) continuous. But we may switch to a topology which is better \emph{adapted} to the problem at hand. Namely the two-timepoint $\AWp$-topology.
  In this case the space for the first timepoint is $\Ys_1 \times \Pr{\Ys_2}$ and that for the second is $\Xs_1 \times \Pr{\Xs_2}$. In effect that means that instead of $\gamma \in \Cpl \mu \nu$ we are now looking at $\gamma' \in \FunP { \Ys_1 \times \Pr {\Ys_2}} {\Pr{ \Xs_1 \times \Pr {\Xs_2}}}$. The function that we are optimizing over can be written as
  \begin{align*}
    \hat C := \gamma' &\mapsto \E[\gamma']{C(\Yr_1,\YYr,\VXXr)}
    \intertext{where}
    C(y_1, \hat y_2, \xi) &= \int \D(x_1,y_1) \d\xi(x_1,\noarg) + \W { \bary(\marg \xi {\Pr {\Xs_2}})} {\hat y_2} \\
    \bary(\lambda) & = \int x \d\lambda(x)
  \end{align*}
  $C$ is a continuous function and so is $\hat C$. Now  $\diso{\Ys_1 \times \Pr{\Ys_2}}\pa{\Cpl \mu \nu}$ is not compact, but 
  \begin{multline*}
    \set<\Big>{\gamma' \in \Pr { \Ys_1 \times \Pr {\Ys_2} \times \Pr{ \Xs_1 \times \Pr {\Xs_2}}}}[ \\ \marg {\gamma'} {\Ys_1 \times \Pr {\Ys_2}} = \nu \, , \quad \marg{\undo {\Ys_1 \times \Pr {\Ys_2}} (\gamma')}{\Pr{ \Xs_1 \times \Pr {\Xs_2}}} = \mu]
  \end{multline*}
  is. So we can find a minimizer $\gamma'$ of $\hat C$ in this set. To return to $\Cpl \mu \nu$, or more precisely $\diso{\Ys_1 \times \Pr{\Ys_2}}\pa{\Cpl \mu \nu}$, we can send $\gamma'$ to the distribution $\gamma''$ of $(\Yr_1,\YYr,\E[\gamma']{\VXXr}[\VYr])$. Because $C$ is continuous and convex in its last argument and by (the conditional version of) Jensens inequality (which could again be proved \enquote*{by hand} here) $\hat C(\gamma'') \leq \hat C(\gamma')$. $\undo { \Ys_1 \times \Pr {\Ys_2}} (\gamma'')$ is the sought after minimizer of \eqref{eq:WcLifted}.
\end{proof}

\begin{lemma}
  \label{lem:Hausdorff}
  Let $\mu, \nu \in \Pr {\Xs_1 \times \Pr {\Xs_2}}$.
  Then $\WWc \mu \nu = \WWc \nu \mu = 0$ implies $\mu = \nu$.
\end{lemma}
\begin{proof}
  Call
  \begin{align*}
    \XXs & := \Xs_1 \times \Pr {\Xs_2} &
    \YYs & := \Ys_1 \times \Pr {\Ys_2} &
    \ZZs & := \Zs_1 \times \Pr {\Zs_2} \fullstop
  \end{align*}
  To have labels for our spaces, see $\mu, \nu$ as $$\mu \in \Pr {\XXs}\,, \quad\nu \in \Pr {\YYs}\,, \quad\mu \in \Pr {\ZZs} \fullstop $$

  Let $\gamma \in \Cpl \mu \nu \subseteq \Pr { \XXs \times \YYs }$ s.t.\ $\E[\gamma]{\D(\Xr_1, \Yr_1)^p} + \E[\gamma]{\W {\E[\gamma]{\XXr}[\VYr]} {\YYr} ^p} = 0$.

  Let $\eta \in \Cpl \nu \mu \subseteq \Pr { \YYs \times \ZZs} $ s.t.\ $\E[\eta]{\D(\Yr_1, \Zr_1)^p} + \E[\eta]{\W {\E[\eta]{\YYr}[\VZr]} {\ZZr} ^p} = 0$.

  \dontsmash
  All the following considerations happen under $\displaystyle \gamma \dpro { \YYs } \eta$.\dosmash{}
  Clearly, $\Zr_1 = \Yr_1 = \Xr_1$ a.s.

  Moreover, because $\E {\XXr}[\VYr,\VZr] = \E {\XXr}[\VYr]$, the random variables $\ZZr, \YYr, \XXr$ form a martingale w.r.t.\ the filtration generated by $\VZr, \VYr, \vec X$. The distribution of $\ZZr$ is equal to the distribution of $\XXr$. Both of these statements are also true if we integrate some bounded measurable function w.r.t.\ our random variables, i.e.\ for any bounded measurable $f : \Xs_2 \rightarrow \R $ we have that $\int f \d\ZZr, \int f \d\YYr, \int f \d\XXr$ is a martingale and that the distribution of $\int f \d\ZZr$ is equal to the distribution of $\int f \d\XXr$. But this means that we must have $\int f \d\ZZr = \int f \d\YYr = \int f \d\XXr$ a.s.\ (Lemma \ref{lem:constMartingale} below). As this is true for all $f$ from a countable generator of the sigma-algebra on $\Xs_2$, we have $\ZZr = \YYr = \XXr$ a.s.
\end{proof}
\begin{lemma}
  \label{lem:constMartingale}
  Let $X_1, X_2, X_3$ be a bounded martingale over $\R$. If the distribution of $X_1$ is equal to the distribution of $X_3$ then $X_1 = X_2 = X_3$ a.s.
\end{lemma}
\begin{proof}
  This is a consequence of the strict version of Jensen's inequality applied to any everywhere strictly convex function. (Take for example $x \mapsto x^2$.)
\end{proof}

\begin{remark}
    The reason we took the detour of turning our probability-measure-valued martingale into a family of martingales on $\R$ and arguing on these is because this way we avoid having to exhibit a continuous, everywhere strictly convex function on $\Pr {\Xs_2}$.
\end{remark}

As a reminder:
\begin{definition}
  For $\mu, \nu \in \Pr {\Xs_1 \times \Pr {\Xs_2}}$,
  \begin{align*}
    \WWca \mu \nu := \max(\WWc \mu \nu , \WWc \nu \mu) \fullstop
  \end{align*}
\end{definition}

\begin{theorem}
  \label{thm:Wca}
  $\WWcac$ is a metric on $\Pr {\Xs_1 \times \Pr {\Xs_2}}$ satisfying $$\WWca \mu \nu \leq \WWb \mu \nu \fullstop$$
\end{theorem}
\begin{proof}
  This follows from Lemma \ref{lem:WcTriangle}, Lemma \ref{lem:Hausdorff} and Lemma \ref{lem:WcleqWb}.
\end{proof}

\begin{remark}
  As outlined at the beginning of this section, {\color{black}and thanks to Theorem \ref{thm:Wca},}  we now know enough to conclude that the topology induced by $\Wcac$ is equal to the topology induced by $\Wbc$, {\color{black}in case both $\Xs_1$ and $\Xs_2$ were compact.}
  The non-compact case is not much harder. {\color{black}We now proceed to settle this case: For this} we need the following lemma.
\end{remark}

\begin{lemma}
  \label{lem:contraction}
  The map
  \begin{align*}
    \undo {\Xs_1} : \Pr {\Xs_1 \times \Pr {\Xs_2}} \rightarrow \Pr {\Xs_1 \times \Xs_2}
  \end{align*}
  is a contraction when we equip the source space with $\WWcac$ and the target space with $\Wa$.
  More specifically for $\mu, \nu \in \Pr {\Xs_1 \times \Pr {\Xs_2}}$
  \begin{align}
    \label{eq:contraction}
    \W {\undo {\Xs_1} (\mu)} {\undo {\Xs_1} (\nu)} \leq \WWc \mu \nu \fullstop
  \end{align}
\end{lemma}
\begin{proof}
  We prove the second statement.
  Let $\mu \in \Pr \XXs$, $\nu \in \Pr \YYs$.
  Given $\gamma \in \Cpl \mu \nu$ and $\epsilon > 0$ the task is to find $\gamma' \in \Cpl {\undo {\Xs_1} \mu} {\undo {\Ys_1} \nu}$ s.t.\
  \begin{align}
    \label{eq:contrtarget}
    \E[\gamma']{\D(\Xr_1,\Yr_1)^p + \D(\Xr_2,\Yr_2)^p} \leq \E[\gamma]{\D(\Xr_1, \Yr_1)^p} + \E[\gamma]{\W {\E[\gamma]{\XXr}[\VYr]} {\YYr} ^p} + \epsilon \fullstop
  \end{align}

  We take inspiration from the discussion at the beginning of this section.
  Let $\Xi : \XXs \times \YYs \rightarrow \Pr {\Xs_2 \times \Ys_2}$ be a measurable selector satisfying
  \begin{align*}
    & \Xi \in \Cpl {\E[\gamma]{\XXr}[\VYr]} {\YYr} \quad \gamma\text{-a.s.\ and} \\
    & \E[\Xi]{\D(\Xr_2,\Yr_2)^p} \leq \W {\E[\gamma]{\XXr}[\VYr]} {\YYr}^p + \epsilon \quad \gamma\text{-a.s.}
  \end{align*}
  The obvious choice for $\gamma'$, namely $f \mapsto \E[\gamma]{\E[\Xi]{f(X_1,X_2,Y_1,Y_2)}}$ will not work because in general it gets the relationship between $X_1$ and $X_2$ wrong, i.e.\ its first marginal may not be $\undo {\Xs_1} (\mu)$.
  Instead we again define $\gamma_L \in \Pr {\Xs_1 \times \Xs_2 \times \Ys_1}$ and $\gamma_R \in \Pr {\Xs_2 \times \Ys_1 \times \Ys_2}$ and set $\gamma' := \gamma_L \dpro {\Xs_2, \Ys_1} \gamma_R$.
  \begin{align*}
    \gamma_L & := f \mapsto \E<\big>[\gamma]{\E[\XXr]{f(\Xr_1,\Xr_2,\Yr_1)}} \\
    \gamma_R & := f \mapsto \E[\gamma]{\E[\Xi]{f(\Xr_2,\Yr_1,\Yr_2)}}
  \end{align*}
  Clearly, if we can actually define $\gamma'$ as announced, then \eqref{eq:contrtarget} will hold, because then
  \begin{align*}
    \E[\gamma']{\D(\Xr_1,\Yr_1)} & = \E<\big>[\gamma]{\E[\XXr]{\D(\Xr_1,\Yr_1)}} = \E[\gamma]{\D(\Xr_1,\Yr_1)} \\
    \E[\gamma']{\D(\Xr_2,\Yr_2)} & = \E[\gamma]{\E[\Xi]{\D(\Xr_2,\Yr_2)}} \leq \E[\gamma]{\W {\E[\gamma]{\XXr}[\VYr]} {\YYr} ^p} + \epsilon \fullstop
  \end{align*}
  It remains to check that $\gamma_L$ and $\gamma_R$ can actually be composed, i.e.\ that $(\Xr_2,\Yr_1)$ has the same distribution under $\gamma_L$ and $\gamma_R$.
  \begin{multline*}
    \E[\gamma_R]{h(\Xr_2,\Yr_1)} = \E[\gamma]{\E[\Xi]{h(\Xr_2,\Yr_1)}} = \E[\gamma]{\E[{ \E[\gamma]{\XXr}[\VYr] }]{h(\Xr_2,\Yr_1)}} = \\ \E[\gamma]{\E<\Big>[\gamma]{ \E[\XXr]{h(\Xr_2,\Yr_1)} }[\VYr] } = \E[\gamma]{ \E[\XXr]{h(\Xr_2,\Yr_1)} } = \E[\gamma_L]{h(\Xr_2,\Yr_1)}
  \end{multline*}
  The step in the middle has its own Lemma \ref{lem:EEExpectation} below.
\end{proof}

\begin{lemma}
  \label{lem:EEExpectation}
  Let $\P$ be a probability on $\Pr \Xs \times \Ys$, for Polish spaces $\Xs, \Ys$. Let $h : \Xs \times \Ys \rightarrow \R$ be a measurable function. Then
  \begin{align*}
    \E [{ \E{\smash{\hat X}}[Y] }] { h(X,Y) } = \E { \E[\hat X] { h(X,Y) } }[Y] \quad \P\text{-a.s.,}
  \end{align*}
  where $\mathbb E$ without superscript is the (conditional) expectiation w.r.t.\ $\P$ and $\hat X$ is the projection onto $\Pr \Xs$.
\end{lemma}
 Note that $X$ is on both sides introduced by the expectation operator which carries a superscript, while $Y$ may on both sides be interpreted as coming from the outermost context. On the right hand side $Y$ may also be seen as having been introduced by the outermost conditional expectation operator. (As this operator conditions on $Y$ this is the same thing.)
\begin{proof}
  Both sides are clearly $Y$-measurable.
  We prove that for $h(x,y) = f(x) g_1(y)$, multiplying by $g_2(Y)$ and taking expectation gives the same result.
  By definition of the conditional expectation
  \begin{align*}
    \E { \E{\smash{\hat X}}[Y] g(Y) } = \E { \hat X g(Y) } \fullstop
  \end{align*}
  Applying the continuous linear function $\gamma \mapsto \E[\gamma]{f(X)}$ this gives
  \begin{align*}
    \E { \E[{\E{\smash{\hat X}}[Y]}]{f(X)} g(Y) } = \E { \E[\smash{\hat X}]{f(X)} g(Y) } \fullstop
  \end{align*}

  Again by the definition of the conditional expectation:
  \begin{multline*}
    \E {\E { \E[\hat X] { f(X)g_1(Y) } }[Y] g_2(Y)} = 
    \E { \E[\hat X] { f(X)g_1(Y) } g_2(Y)} =  \\ \E { \E[\hat X] { f(X) } g_1(Y) g_2(Y) } = \E { \E[{\E{\smash{\hat X}}[Y]}]{f(X)} g_1(Y) g_2(Y) } = \\ \E { \E[{\E{\smash{\hat X}}[Y]}]{f(X) g_1(Y)} g_2(Y) }
  \end{multline*}
  where for the third equality we plugged in the previous equation.
\end{proof}

{%
\begin{proof}[Alternative proof of Lemma \ref{lem:contraction} when $\Xs_1$ has no isolated points]
\dosmash%
  When the space $\Xs_1$ has no isolated points one can show that the space $\FunP {\Xs_1} {\Pr {\Xs_2}}$ is dense in $\Pr {\Xs_1 \times \Pr {\Xs_2}}$. This allows for a shorter proof of Lemma \ref{lem:contraction}:

  By the original definition \eqref{eq:WcDef} of $\Wcau$ on the space $\Pr {\Xs_1 \times \Xs_2}$ the inequality \eqref{eq:contraction} holds on $\FunP {\Xs_1} {\Pr {\Xs_2}} \times \FunP {\Xs_1} {\Pr {\Xs_2}}$. Both $\Wcau$ and $(\mu,\nu) \mapsto \W {\undo {\Xs_1} (\mu)} {\undo {\Xs_1} (\nu)}$ are uniformly continuous on $\Pr {\XXs} \times \Pr {\XXs}$ w.r.t.\ some product metric of $\WWbc$ with itself. $\FunP {\Xs_1} {\Pr {\Xs_2}}$ is dense in $\Pr {\XXs}$, and therefore $\FunP {\Xs_1} {\Pr {\Xs_2}} \times \FunP {\Xs_1} {\Pr {\Xs_2}}$ is dense in $\Pr {\XXs} \times \Pr {\XXs}$. This implies that \eqref{eq:contraction} holds on all of $\Pr {\XXs} \times \Pr {\XXs}$.
\end{proof}
}%

\begin{theorem}
  \label{thm:CausalAntiEqBi2}
  The topology induced by $\WWcac$ on $\Pr { \Xs_1 \times \Pr {\Xs_2} }$ is equal to the toplogy induced by $\WWbc$ on that space.
\end{theorem}
\begin{proof}
  As both topologies are metric and therefore first-countable we may argue on sequences.
  Let $(\mu_n)_n$ be a sequence in $\Pr { \Xs_1 \times \Pr {\Xs_2} }$. As $\WWca {\mu_n} \mu \leq \WWb {\mu_n} \mu$, if $(\mu_n)_n$ converges to $\mu$ w.r.t.\ $\WWbc$ it also converges to $\mu$ w.r.t.\ $\WWcac$.

  Now assume that a sequence $(\mu_n)_n$ in $\Pr { \Xs_1 \times \Pr {\Xs_2} }$ converges to $\mu$ w.r.t.\ $\WWcac$.
  We will show that every subsequence of $(\mu_n)_n$ has a subsequence which converges to $\mu$ w.r.t.\ $\WWbc$.
  {\color{black} 
  Note that convergence of $(\mu_n)_n$ w.r.t.\ $\WWcac$ implies that the set $K := \set{ \mu_n }[n \in \N]$ is relatively compact w.r.t.\ the topology induced by $\WWcac$.}
  As $\undo {\Xs_1}$ is continuous as a map from $\Pr { \Xs_1 \times \Pr {\Xs_2} }$ with the topology induced by $\WWcac$ to $\Pr {\Xs_1 \times \Xs_2}$ with the toplogy induced by $\Wa$ (Lemma \ref{lem:contraction}), we have that $\undo {\Xs_1} [K] = \set{ \undo {\Xs_1} (\mu_n) }[n \in \N]$ is also relatively compact.
  By Lemma \ref{lem:CompactnessLemma}/\cite[Lemma 3.3]{ModulusOfContinuity} this implies that $K$ is relatively compact in $\Pr {\Xs_1 \times \Pr {\Xs_2}}$ with the topology induced by $\WWbc$.
  Now let $(\mu_{n_k})_k$ be some subsequence of $(\mu_n)_n$.
  As $K$ is relatively compact we can find a subsequence $(\mu_{n_{k_j}})_j$ of $(\mu_{n_k})_k$, which converges w.r.t.\ $\WWbc$ to some $\mu' \in \Pr {\Xs_1 \times \Pr {\Xs_2} }$.
  As $\WWca {\mu_{n_{k_j}}} {\mu'} \leq \WWb {\mu_{n_{k_j}}} {\mu'}$ this sequence also converges to $\mu'$ w.r.t.\ $\WWcac$. But $(\mu_{n_{k_j}})_j$ also converges to $\mu$ w.r.t.\ $\WWcac$.
  Because the topology induced by $\WWcac$ is Hausdorff (Lemma \ref{lem:Hausdorff}), we must have $\mu' = \mu$, i.e.\ $(\mu_{n_{k_j}})_j$ converges to $\mu$ w.r.t.\ $\WWbc$.
\end{proof}

Now we return to the general case of $N$ time-points.

\begin{theorem}
  \label{thm:CausalAntiEqBiN}
  The topology induced by $\Wcac$ on $\Pr {\Xb}$ is equal to Hellwig's $\Wp$-information topology and to the topology induced by $\AWp$.
\end{theorem}
\begin{proof}
  As every bicausal transport plan between $\mu$ and $\nu$ can be interpreted as a causal transport plan from $\mu$ to $\nu$ and also as a causal transport plan from $\nu$ to $\mu$ we have that $\Wca \mu \nu \leq \Wbb (\mu, \nu)$. This means that the identity from $\Pr \Xb$ with the topology induced by $\AWp$ to $\Pr \Xb$ with the topology induced by $\Wcac$ is continuous.
  For the other direction we show that the identity from $\Pr \Xb$ with the topology induced by $\Wcac$ to $\Pr \Xb$ with the $\Wp$-information topology is continuous, i.e.\ we show that each of the maps
  \begin{align*}
    \disint {\Xb^t} {\Xb_{t+1}} = \I_t : \Pr \Xb \rightarrow \FunP {\Xb^t} {\Pr{ \Xb_{t+1} }} 
  \end{align*}
  is continuous when $\Pr \Xb$ gets the topology induced by $\Wcac$.
  
  If $\mu, \nu \in \Pr {\Xb}$ and $\gamma \in \Cpl \mu \nu$ is causal, then, in particular, $\gamma$ is \enquote*{causal at the timestep from $t$ to $t+1$}, i.e.\ $\gamma$ is causal when regarded as a coupling between $\mu, \nu \in \Pr {\Xb^t \times \Xb_{t+1}}$.
  This means that if we define $\Wcaa$ like $\Wcac$, but only require causality based on the decomposition of $\Xb$ as $\Xb^t \times \Xb_{t+1}$, then $\Wcaa (\mu,\nu) \leq \Wca \mu \nu$, i.e.\ the identity from $\Pr \Xb$ with the topology induced by $\Wcac$ to $\Pr \Xb$ with the topology induced by $\Wcaa$ is continuous. By Theorem \ref{thm:CausalAntiEqBi2} the map
  \begin{align*}
    \disint {\Xb^t} {\Xb_{t+1}} : \Pr {\Xb^t \times \Xb_{t+1}} \rightarrow \FunP {\Xb^t} {\Pr {\Xb_{t+1}}}
  \end{align*}
  is continuous when we equip $\Pr {\Xb^t \times \Xb_{t+1}}$ with the topology induced by $\Wcaa$. Now $\I_t$ is continuous as a composite of continuous maps.
\end{proof}

\section{Aldous' extended weak convergence}
\label{sec:Aldous}

In this section we show that Aldous extended $\Wp$-/weak topology is equal to Hellwig's ($\Wp$-)information topology.

We recall and paraphrase here the definition, already given in the introduction, of Aldous' topology.

\newcommand{\dmu}[1]{\mu_{#1}}
\newcommand{\x}[2][i]{(x_{#1})_{#1=1}^{#2}}
\newcommand{\dirac}[1]{\delta_{#1}}
\begin{definition}
  Given $\mu \in \Pr \Xb$ let $\dmu {\x j}$ be the value of a (classical) disintegration of $\mu$ w.r.t.\ the first $j$ coordinates at $\x j$. (By convention $\dmu {\x 0} = \mu$).
  Define
  \begin{align*}
    \Extw & : \Pr \Xb \rightarrow \Pr { \Xb \times \prod_{j=0}^N \Pr \Xb } \\
    \Extw(\mu) & := \push{ \x N \mapsto \pa{ \x N , \pa{ \dirac {\x j} \mprod \dmu {\x j} }_{j=0}^N } } (\mu) \fullstop
  \end{align*}
  The \emph{extended $\Wp-$/weak topology} on $\Pr \Xb$ is the initial topology w.r.t.\ $\Extw$.
\end{definition}
\begin{remark}
  Reasonable people may disagree about whether the most faithful / useful transcription of Aldous' definition should include the factors $j=0$ and $j=N$ in the above product of spaces. When including $j=N$, as we did, one has to interpret $\dirac {\x N} \mprod \dmu {\x N}$ simply as $\dirac {\x N}$. We leave it as an exercise to the reader to check that either or both may be dropped in the definition of $\Extw$ without affecting the resulting topology on $\Pr X$.
\end{remark}

\newcommand{\Alde}[1]{\mathcal A'_{#1}}
\newcommand{\Ald}{\mathcal A}
\begin{theorem}
  The ($\Wp$-)extended weak topology is equal to the ($\Wp$-)information topology.
\end{theorem}
\begin{proof}
  We construct continuous maps
  \begin{align*}
    \Alde k & : \Pr { \Xb \times \prod_{j=0}^N \Pr \Xb } \rightarrow \Pr { \Xb^k \times \Pr {\Xb_{k+1}} } \\
    \Ald & : \prod_{k=1}^{N-1} \FunP {\Xb^k} {\Pr{ \Xb_{k+1} }} \rightarrow \Pr { \Xb \times \prod_{j=0}^N \Pr \Xb } \\
  \end{align*}
  such that
  \begin{align*}
    \Alde k \circ \Extw & = \I_k \\
    \Ald \circ \I & = \Extw \fullstop
  \end{align*}
  The first equality above implies that the identity on $\Pr \Xb$ is continuous from the extended weak topology to the information topology, the second implies that it is continuous in the other direction.

  $\Alde k$ is very simple. We just need to select the right factors and then discard the unnecessary $\dirac {\x k}$ part of the measure component. Formally
  \begin{align*}
    \Alde k := \push{ \big( \x N , (\nu_j)_{j=0}^N \big) \mapsto \big( \x k , \marg {{\nu_{k}}} {\Xb_{k+1}} \big) } \text{ ,}
  \end{align*}
  which is cleary continuous.

  We construct $\Ald$ recursively, by constructing as a composite of continuous maps
  \begin{align*}
    \Ald^m & : \prod_{k=1}^{N-1} \FunP {\Xb^k} {\Pr{ \Xb_{k+1} }} \rightarrow \Pr { \Xb^m \times \prod_{j=0}^m \Pr \Xb }
  \end{align*}
  satisfying
  \begin{align}
    \label{eq:aldindinv}
    \Ald^m(\I(\mu)) = \push{ \x N \mapsto \big( \x m , ( \dirac{\x k} \mprod \dmu{\x k} )_{k=0}^m \big) }(\mu) \fullstop
  \end{align}
  $\Ald^0\pa{ (\nu_k)_{k=1}^{N-1} } := \dirac {\undo {\Xs_1} (\nu_1)}$. We need the helper functions
  \begin{align*}
    h_m & : \FunP {\Xb^m} {\Pr {\Xb_{m+1}}} \rightarrow \FunP {\Xb^m} {\Pr \Xb} \\
    h_m & := \push{ (\x m, \rho) \mapsto (\x m , \dirac{\x m} \mprod \rho) } \fullstop
  \end{align*}
  Given $\Ald^m$ satisfying the induction hypothesis we set
  \begin{align*}
    \Ald^{m+1}\pa{ (\nu_k)_{k=1}^{N-1} } := \push{ s_{m+1} } \pa{
      \Ald^m\pa{ (\nu_k)_{k=1}^{N-1} } \dpro{\Xb^m} h_{m+1}(\nu_{m+1})
    }
  \end{align*}
  where $s_{m+1}$ is the obvious permutation of the coordinates to get the factors into the right order.
  $\Ald^{m+1}$ is continuous because by \cite[Lemma 4.1]{ModulusOfContinuity} $\dpro{\Xb^m}$ is continuous when one of the arguments is an element of some $\FunP \Bs \Cs$. That \eqref{eq:aldindinv} still holds for $m+1$ is a straightforward calculation.
  This way we get to $\Ald^{N-1}$. Finally, set
  \begin{align*}
    \Ald\pa{ (\nu_k)_{k=1}^{N-1} } := \push{ s_N } \pa{ %
      \Ald^{N-1}\pa{ (\nu_k)_{k=1}^{N-1} } \dpro{\Xb^{N-1}} \diso{\Xs_1}(\nu_1) %
    } \\
    \intertext{where}
    s_N\pa{ \pa{\x {N-1}, (\rho_j)_{j=1}^{N-1}, x_N} } := \pa{ \x {N}, (\rho_j)_{j=1}^{N-1}, \dirac{(\x {N})}} \fullstop
  \end{align*}
\end{proof}

\section{Bounded vs unbounded metrics}
\label{sec:convergence.moments}

Because we will need it in the next section we interject here a proof of Lemma \ref{lem:convergence.moments}, which we restate below.

{\def\Omega{{\Xb}}
\convergencemoments*
}

\begin{proof}[Proof of Lemma \ref{lem:convergence.moments}]
  We provide the proof only for Hellwig's topology, i.e.\ \ref{it:Hellwig} of Theorem \ref{thm:main} and Theorem \ref{thm:main_simple}, respectively.
  As we have already seen in the previous sections, the topologies \ref{it:SCW}--\ref{it:Aldous} are equivalent topologies, and the result therefore carries over to them. The ($\Wp$-)optimal stopping topology, \ref{it:optstop}, is treated below.
  It is clear that convergence w.r.t.\ $\Wp$-information topology implies convergence in Hellwig'g information topology plus convergence of $p$-th moments.
  For the reverse implication, let $1\leq t\leq N-1$, and denote by $\mathcal{A}:=\overline{\mathcal{X}}^t$ the first $t$ and by $\mathcal{B}:=\overline{\mathcal{X}}_{t+1}$ the last $N-t$ coordinates.
  Now assume that $(\mu_n)_n$ converges to $\mu$ in Hellwig's information topology and that the $p$-th moments converge.
  The classical (not adapted) version of the very lemma we prove here implies that $\mu_n\to\mu$ in $\Wp$; in particular $K:=\{\mu_n: n\}\subset \Prp{\mathcal{A}\times\mathcal{B}}$ is relatively compact.
  Lemma \ref{lem:CompactnessLemma} (or really \cite[Lemma 3.3]{ModulusOfContinuity}/\cite[Lemma 2.6]{BaBePa18}) therefore guarantees that $\dis_\mathcal{A}^\mathcal{B} [K] \subset \Prp{\mathcal{A}\times\Prp{\mathcal{B}}}$ is relatively compact.

  Every subsequence of $(\disint \A \B (\mu_n))_n$ therefore has a subsequence $(\disint \A \B (\mu_{n_k}))_k$ which converges w.r.t.\ the topology on $\Prp {\A \times \Prp \B}$ (i.e.\ the one coming from nested Wasserstein metrics) to some $\mu' \in \Prp {\A \times \Prp \B}$. Because convergence in $\Prp {\A \times \Prp \B}$ is stronger than convergence in $\Pr {\A \times \Pr \B}$ (i.e.\ in the nested \emph{weak} sense) we must also have $\disint \A \B (\mu_{n_k}) \overset{k}{\to} \mu'$ in $\Pr {\A \times \Pr \B}$. But also, by assumption, $\disint \A \B (\mu_{n_k}) \overset{k}{\to} \disint \A \B (\mu)$ in $\Pr {\A \times \Pr \B}$ and therefore $\mu' = \disint \A \B (\mu)$.
\end{proof}

\section{Optimal Stopping}
\label{sec:OptStop}

In this section we investigate the relation between the ($\Wp$-)optimal stopping topology and the adapted Wasserstein topology.
Lemma \ref{lem:stopping.controlled.by.causal} states that the topology induced by $\AWp$ (\ref{it:AW} of Theorem \ref{thm:main}) is finer than the $\Wp$-optimal stopping topology.
Lemma \ref{lem:stopping.finer.than.information} states that the $\Wp$-optimal stopping topology is finer than the $\Wp$-information topology (\ref{it:Hellwig} of Theorem \ref{thm:main}).
This will finish the proof of Theorem \ref{thm:main}.

Recall that 
\begin{align*}
  v^L(\mu):=\inf\set{\E[\mu] {L_\tau(X))}: 0\leq \tau\leq N\mbox{ is a stopping time} }
\end{align*}
for $L=(L_t)_{t=0}^N\in AC_p(\Omega)$.

\begin{lemma}
\label{lem:stopping.controlled.by.causal}
  Let $L\in AC_p(\Omega)$.
  Then $\mu\mapsto v^L(\mu)$ is continuous w.r.t.\ $\AWp$.
  In fact, one has
  \begin{align}
  \label{eq bound opt stopp}
  |v^L(\mu)-v^L(\nu)|\leq \inf\set{ \E[\pi] { \max_{0\leq t\leq N} |L_t(X)-L_t(Y)| } : \,\, \pi\in\cplba(\mu,\nu) }.
  \end{align}
  for every $\mu,\nu\in\Prp{\Omega}$.
\end{lemma}
\begin{proof}
  Let $\mu,\nu\in\Prp{\Omega}$ and assume that $v^L(\mu)\leq v^L(\nu)$.
  Moreover, let $\pi\in\cplba(\mu,\nu)$ and $\varepsilon>0$ be arbitrary, and fix a stopping time $\tau$ satisfying $\E[\nu]{L_{\tau}(Y)}\leq v^L(\nu)+\varepsilon$.
  For $u\in[0,1]$ define
  \begin{align*}
  \sigma(X,u)
  &:=\inf\{ t\in\{0,\cdots,T\}: \pi(\tau(Y)\leq t| X)\geq u\} \\
  &\phantom{:}= \inf\{ t\in\{0,\cdots,T\} : \pi(\tau(Y)\leq t| X_1,\dots,X_t)\geq u\},
  \end{align*}
  where the equality holds by the properties of stopping times and since $\pi$ is causal.
  We then have that
  \begin{align*}
  \int_{[0,1]}\E[\pi]{ L_{\sigma(X,u)}(X)} \d u \
  &= \sum_{t=0}^T \int_{[0,1]}\E[\pi]{ L_t(X)1_{\pi(\tau(Y)\leq t| X)\geq u >\pi(\tau(Y)\leq t-1 | X) } } \d u \\
  & =\sum_{t=0}^T \E[\pi] {L_t(X)1_{\tau(Y)=t} }
  = \E[\pi] {L_{\tau(Y)}(X) }.
  \end{align*}
  As further $\sigma(\cdot,u)$ is a stopping time for every fixed $u\in[0,1]$ one has $v^L(\mu)\leq \int_{[0,1]}\E[\pi] {L_{\sigma(X,u)}(X) } \d u$ and therefore 
  \begin{align*}
  v^L(\mu)- v^L(\nu) 
  &\leq \E[\pi] {L_{\tau(Y)}(X) - L_{\tau(Y)}(Y) } +\varepsilon\\
  &\leq \E[\pi] { \max_{0\leq t\leq N} |L_t(X) - L_t(Y) | } +\varepsilon.
  \end{align*}
  Changing the role of $\mu$ and $\nu$ and using that $\varepsilon>0$ and $\pi\in\cplba(\mu,\nu)$ was arbitrary yields \eqref{eq bound opt stopp}.

  Now assume that
  $\AWp(\mu_n,\mu) \to 0$ 
  and that $\pi_n \in \Cpl {\mu_n} \mu$ is less than $1/n$ away from attaining the infimum $\AWp(\mu_n,\mu)$.
  Then $\Wp(\pi_n, \pi) \to 0$, where $\pi \in \Cpl \mu \mu$ is the identity coupling $\push {\id \Omega \funcomma \id \Omega} (\mu)$ of $\mu$. (A coupling between $\pi_n$ and $\pi$ is given by $\push {(x,y) \mapsto (x,y,y,y)}(\pi_n)$.) Because $(x,y) \mapsto \max_{0\leq t\leq N} |L_t(x)-L_t(y)|$ is a continuous function of growth of at most order $p$, we get that
  \begin{align*}
    \E[\pi_n] { \max_{0\leq t\leq N} |L_t(X)-L_t(Y)| } \to \E[\pi] { \max_{0\leq t\leq N} |L_t(X)-L_t(Y)| } = 0 \fullstop
  \end{align*}
  Together with \eqref{eq bound opt stopp} this implies that $v^L$ is continuous w.r.t.\ $\AWp$.
\end{proof}

\begin{remark}
  The above proof reveals that if $L_t$ is Lipschitz with constant $c>0$ for every $t$, then $|v^L(\mu)-v^L(\nu)|\leq c\, \SWA_1(\mu,\nu)$.
\end{remark}

In order to show that the optimal stopping topology is finer than the $\Wp$-information topology, we need to make a few preparations.

\begin{lemma}\label{lemv con det 1}
  Let $\mathcal{A}$ be a Polish space. 
  Then the family 
  \begin{align}
  \label{eq:convergence.set.1}
  \set{ \Pr{\mathcal{A}} \ni \mu \mapsto G\left(\int_\mathcal{A} h_1 \d\mu,\dots,\int_\mathcal{A} h_L \d\mu \right ) :
  \begin{array}{l} 
  L\in\mathbb N,G\in C_b(\mathbb R^L)\\
  (h_i)_{i\leq L}\subset C_b(\mathcal{A})
  \end{array}  }
  \end{align}
  is convergence determining for the weak topology on $\Pr{\Pr{\mathcal{A}}}$, that is, a sequence of probability measures $(\mu_n)_n$ in $\Pr{\Pr{\mathcal{A}}}$ converges weakly to a probability measure $\mu\in \Pr{\Pr{\mathcal{A}}}$ if and only if $\int F\d\mu_n\to\int F\d\mu$ for all $F$ in \eqref{eq:convergence.set.1}.
\end{lemma}

This follows from the Stone-Weierstrass theorem in case of compact $\mathcal{A}$ and readily extends to general Polish spaces e.g.\ via Stone-Čech compactification.

\begin{lemma}
\label{lemv con det 2}
  Let $\mathcal{A}$ be a Polish space.
  The family of functions
  \begin{align}
  \label{eq:convergence.set.2}
    \set{  \mu\mapsto G\left(\int_{\mathcal{A}} h \d\mu\right)\,:\,\, h\in C_b(\mathcal{A}),  G \in C_b(\R)  }
  \end{align}
  is convergence determining for the weak topology on $\Pr{\Pr{\mathcal{A}}}$.
\end{lemma}
\begin{proof}
  Let $L$, $G$, and $(h_i)_{i\leq L}$ as in \eqref{eq:convergence.set.1}.
  Moreover, let $m\in\mathbb{R}$ such that $|h_i|\leq m$ for all $1\leq i\leq L$ and define $I:=[-m,m]^L$.
  Then $I\subset\mathbb{R}^L$ is compact and satisfies 
  \[\left(\int h_1d\mu,\dots,\int h_L\d\mu\right)\in I\quad\text{for all }\mu\in \Pr{\mathcal{A}}.\] 

  Let $\sigma\colon\mathbb{R}\to\mathbb{R}$ be some fixed bounded continuous sigmoid function such as $\sigma(r)=(1+e^{-r})^{-1}$ or $\sigma(r)=\max(0,\min(r,1))$.

  By the universal approximation result of Cybenko \cite[Theorem 2]{Cy89}, the set 
  \[\set{ x\mapsto  \sum_{i=1}^m u_i \sigma(v_i\cdot x+w_i): 
  \begin{array}{l}
  m\in\mathbb N,(u_i)_{i\leq m}\subset \mathbb{R},\\
  (v_i)_{i\leq m}\subset \mathbb{R}^L, (w_i)_{i\leq m}\subset \mathbb{R}
  \end{array} }\]
  is dense in $C(I,\mathbb{R})$ w.r.t.\ the supremum norm.
  As a result, it is enough to replace $G$ in \eqref{eq:convergence.set.1} by functions of the form $x\mapsto\sum_{i=1}^m u_i \sigma(v_i\cdot x+w_i)$. 
  Evaluating the latter function on the vector $x=(\int h_1\d\mu,\dots,\int h_L\d\mu)$ yields
  \begin{align*}
  \sum_{i=1}^m u_i \sigma\left(\sum_{k=1}^L v_i^k\int h_k \d \mu + w_i\right)
  &=\sum_{i=1}^m u_i \sigma\left(\,\int \left (\sum_{k=1}^{L+1} v_i^k h_k \right ) \d \mu\,\right)\\
  &=\sum_{i=1}^m u_i\sigma\left( \int \bar h_i \d \mu \right),
  \end{align*}
  upon defining $v_i^{L+1}:=b_i$, $w_{L+1}:=1$, and finally $\bar h_i:= \sum_{k=1}^{L+1} v_i^k h_k$ for every $i$. 
  The result follows from Lemma \ref{lemv con det 1}.
\end{proof}

\begin{lemma}
\label{lem:stopping.finer.than.information}
  The $\Wp$-optimal stopping topology is finer than the $\Wp$-information topology.
\end{lemma}
\begin{proof}
  The choice $L_T:=-\D(x,x_0)^p-1$ and $L_t:=0$ for $t\neq T$ shows that convergence in the $\Wp$-optimal stopping topology implies convergence of the $p$-th moments.
  Thus, we are left to show that convergence in the optimal stopping topology implies convergence in Hellwig's information topology.
  Then, by the part of Lemma \ref{lem:convergence.moments} which has already been established, we obtain convergence in the $\Wp$-information topology.
  
  Fix $1\leq t\leq N-1$ and denote by $\mathcal{A}:=\overline{\mathcal{X}}^t$ the first $t$ and by $\mathcal{B}:=\overline{\mathcal{X}}_{t+1}$ the last $N-t$ coordinates.
  As $C_b(\mathcal{A})$ is convergence determining for $\Pr{\mathcal{A}}$, and $\{\nu\mapsto G (\int_{\mathcal{B}} h\d\nu):h\in C_b(\mathcal{B}), G \in C_b(\R)\}$ is, by Lemma \ref{lemv con det 2}, convergence determining for $\Pr{\Pr{\mathcal{B}}}$, it follows e.g.\ from \cite[Proposition 4.6 (p.115)]{EtKu09} that
  \begin{align}
  \label{eq:convergence.set.3b}
  \set{ (a,\nu)\mapsto f(a)g\left(\int_\mathcal{B} h(b) \d\nu(b) \right) : f\in C_b(\mathcal{A}),\,g\in C_b(\mathbb R),\,h\in C_b(\mathcal{B}) },
  \end{align}
  is convergence determining for the weak topology on $\Pr{\mathcal{A}\times\Pr{\mathcal{B}}}$. 
  Since $h$ in \eqref{eq:convergence.set.3b} is bounded, one can actually take $g$ in \eqref{eq:convergence.set.3b} to be compactly supported. 
  But a continuous compactly supported function can be approximated uniformly by piecewise linear functions. 
  The latter are linear combinations of functions of the form $z\mapsto\min(c,dz)$ where $c,d\in\mathbb{R}$.
  It therefore follows that  
  \begin{align}
  \label{eq:convergence.set.4}
  \set{ (a,\nu) \mapsto \min\left( f(a)\,,\, \int_\mathcal{B} f(a)h(b) \d\nu(b) \right) : f\in C_b(\mathcal{A}), h\in C_b(\mathcal{B}) },
  \end{align}
  is also convergence determining for the weak topology on $\Pr{\mathcal{A}\times\Pr{\mathcal{B}}}$.
  Let $F$ be a function in \eqref{eq:convergence.set.4}, defined via $f\in C_b(\mathcal{A})$ and $h\in C_b(\mathcal{B})$, and let $m\in\mathbb{R}$ be a bound for $|f|$ and $|h|$.
  Define $L\in AC_p(\Omega)$ via
  \[L_t:=f\circ \overline{X}^t \quad L_T:=(f\circ \overline{X}^t)  \cdot (h\circ \overline{X}_{t+1}) \quad\text{and } L_s:=m+1\text{ for }s\neq t,T. \]
  (Where $\overline{X}^t$ is the projection onto the first $t$ coordinates and $\overline{X}_{t+1}$ is the projection onto the remaining $N-t$ coordinates.)

  By dynamic programming (the Snell-envelope theorem) one has 
  \begin{align*}
    v^L(\mu)&= \E[\mu]{ \min\left( f(\overline{X}^t), \E[\mu]{ f(\overline{X}^t)h(\overline{X}_{t+1}) | \overline{X}^t }\right) }\\
    &=\int_{\mathcal{A}\times\Pr{\mathcal{B}}} F \d(\disint {\mathcal{A}} {\mathcal{B}}(\mu))
  \end{align*}
  for every $\mu\in\Pr{\mathcal{A}\times\mathcal{B}}$.
  This implies that the optimal stopping topology is finer than the initial topology of $\mu\mapsto \int F \d (\disint {\mathcal{A}} {\mathcal{B}}(\mu))$ over $F$ in \eqref{eq:convergence.set.4}.
  As \eqref{eq:convergence.set.4} is convergence determining for the weak topology on $\Pr{\mathcal{A}\times\Pr{\mathcal{B}}}$, the optimal stopping topology is indeed finer than the information topology, and as observed at the beginning of this proof therefore the $\Wp$-optimal stopping topology is finer than the $\Wp$-information topology.
\end{proof}

\bibliography{lib/joint_biblio,lib/own}{}
\bibliographystyle{abbrv}
\makeatletter\@input{lib/ModulusOfContinuity_aux.tex}\makeatother

\end{document}